\documentclass[11pt]{amsart}

\usepackage{txfonts}
\usepackage{amsfonts}
\usepackage{mathrsfs}
\usepackage{latexsym}
\usepackage{graphicx}
\usepackage{amscd,amssymb,amsmath,amsbsy,amsthm}
\usepackage[all, cmtip]{xy}
\usepackage[colorlinks,plainpages,urlcolor=blue]{hyperref}
\usepackage{verbatim}
\usepackage[all,cmtip]{xy}
\usepackage{tikz}
\usetikzlibrary{arrows,calc}
\tikzset{
>=stealth',
help lines/.style={dashed, thick},
axis/.style={<->},
important line/.style={thick},
connection/.style={thick, dotted},
}

\newcommand{\nc}{\newcommand}
\nc{\rnc}{\renewcommand}
\nc{\bb}[1]{{\mathbb #1}}
\nc{\bbA}{\bb{A}}\nc{\bbB}{\bb{B}}\nc{\bbC}{\bb{C}}\nc{\bbD}{\bb{D}}
\nc{\bbE}{\bb{E}}\nc{\bbF}{\bb{F}}\nc{\bbG}{\bb{G}}\nc{\bbH}{\bb{H}}
\nc{\bbI}{\bb{I}}\nc{\bbJ}{\bb{J}}\nc{\bbK}{\bb{K}}\nc{\bbL}{\bb{L}}
\nc{\bbM}{\bb{M}}\nc{\bbN}{\bb{N}}\nc{\bbO}{\bb{O}}\nc{\bbP}{\bb{P}}
\nc{\bbQ}{\bb{Q}}\nc{\bbR}{\bb{R}}\nc{\bbS}{\bb{S}}\nc{\bbT}{\bb{T}}
\nc{\bbU}{\bb{U}}\nc{\bbV}{\bb{V}}\nc{\bbW}{\bb{W}}\nc{\bbX}{\bb{X}}
\nc{\bbY}{\bb{Y}}\nc{\bbZ}{\bb{Z}}
\nc{\mbf}[1]{{\mathbf #1}}
\nc{\bfA}{\mbf{A}}\nc{\bfB}{\mbf{B}}\nc{\bfC}{\mbf{C}}\nc{\bfD}{\mbf{D}}
\nc{\bfE}{\mbf{E}}\nc{\bfF}{\mbf{F}}\nc{\bfG}{\mbf{G}}\nc{\bfH}{\mbf{H}}
\nc{\bfI}{\mbf{I}}\nc{\bfJ}{\mbf{J}}\nc{\bfK}{\mbf{K}}\nc{\bfL}{\mbf{L}}
\nc{\bfM}{\mbf{M}}\nc{\bfN}{\mbf{N}}\nc{\bfO}{\mbf{O}}\nc{\bfP}{\mbf{P}}
\nc{\bfQ}{\mbf{Q}}\nc{\bfR}{\mbf{R}}\nc{\bfS}{\mbf{S}}\nc{\bfT}{\mbf{T}}
\nc{\bfU}{\mbf{U}}\nc{\bfV}{\mbf{V}}\nc{\bfW}{\mbf{W}}\nc{\bfX}{\mbf{X}}
\nc{\bfY}{\mbf{Y}}\nc{\bfZ}{\mbf{Z}}
\nc{\bfa}{\mbf{a}}\nc{\bfb}{\mbf{b}}\nc{\bfc}{\mbf{c}}\nc{\bfd}{\mbf{d}}
\nc{\bfe}{\mbf{e}}\nc{\bff}{\mbf{f}}\nc{\bfg}{\mbf{g}}\nc{\bfh}{\mbf{h}}
\nc{\bfi}{\mbf{i}}\nc{\bfj}{\mbf{j}}\nc{\bfk}{\mbf{k}}\nc{\bfl}{\mbf{l}}
\nc{\bfm}{\mbf{m}}\nc{\bfn}{\mbf{n}}\nc{\bfo}{\mbf{o}}\nc{\bfp}{\mbf{p}}
\nc{\bfq}{\mbf{q}}\nc{\bfr}{\mbf{r}}\nc{\bfs}{\mbf{s}}\nc{\bft}{\mbf{t}}
\nc{\bfu}{\mbf{u}}\nc{\bfv}{\mbf{v}}\nc{\bfw}{\mbf{w}}\nc{\bfx}{\mbf{x}}
\nc{\bfy}{\mbf{y}}\nc{\bfz}{\mbf{z}}

\nc{\mcal}[1]{{\mathcal #1}}
\nc{\calA}{\mcal{A}}\nc{\calB}{\mcal{B}}\nc{\calC}{\mcal{C}}\nc{\calD}{\mcal{D}}
\nc{\calE}{\mcal{E}} \nc{\calF}{\mcal{F}}\nc{\calG}{\mcal{G}}\nc{\calH}{\mcal{H}}
\nc{\calI}{\mcal{I}}\nc{\calJ}{\mcal{J}}\nc{\calK}{\mcal{K}}\nc{\calL}{\mcal{L}}
\nc{\calM}{\mcal{M}}\nc{\calN}{\mcal{N}}\nc{\calO}{\mcal{O}}\nc{\calP}{\mcal{P}}
\nc{\calQ}{\mcal{Q}}\nc{\calR}{\mcal{R}}\nc{\calS}{\mcal{S}}\nc{\calT}{\mcal{T}}
\nc{\calU}{\mcal{U}}\nc{\calV}{\mcal{V}}\nc{\calW}{\mcal{W}}\nc{\calX}{\mcal{X}}
\nc{\calY}{\mcal{Y}}\nc{\calZ}{\mcal{Z}}
\nc{\fA}{\frak{A}}\nc{\fB}{\frak{B}}\nc{\fC}{\frak{C}} \nc{\fD}{\frak{D}}
\nc{\fE}{\frak{E}}\nc{\fF}{\frak{F}}\nc{\fG}{\frak{G}}\nc{\fH}{\frak{H}}
\nc{\fI}{\frak{I}}\nc{\fJ}{\frak{J}}\nc{\fK}{\frak{K}}\nc{\fL}{\frak{L}}
\nc{\fM}{\frak{M}}\nc{\fN}{\frak{N}}\nc{\fO}{\frak{O}}\nc{\fP}{\frak{P}}
\nc{\fQ}{\frak{Q}}\nc{\fR}{\frak{R}}\nc{\fS}{\frak{S}}\nc{\fT}{\frak{T}}
\nc{\fU}{\frak{U}}\nc{\fV}{\frak{V}}\nc{\fW}{\frak{W}}\nc{\fX}{\frak{X}}
\nc{\fY}{\frak{Y}}\nc{\fZ}{\frak{Z}}
\nc{\fa}{\frak{a}}\nc{\fb}{\frak{b}}\nc{\fc}{\frak{c}} \nc{\fd}{\frak{d}}
\nc{\fe}{\frak{e}}\nc{\fFf}{\frak{f}}\nc{\fg}{\frak{g}}\nc{\fh}{\frak{h}}
\nc{\fri}{\frak{i}}\nc{\fj}{\frak{j}}\nc{\fk}{\frak{k}}\nc{\fl}{\frak{l}}
\nc{\fm}{\frak{m}}\nc{\fn}{\frak{n}}\nc{\fo}{\frak{o}}\nc{\fp}{\frak{p}}
\nc{\fq}{\frak{q}}\nc{\fr}{\frak{r}}\nc{\fs}{\frak{s}}\nc{\ft}{\frak{t}}
\nc{\fu}{\frak{u}}\nc{\fv}{\frak{v}}\nc{\fw}{\frak{w}}\nc{\fx}{\frak{x}}
\nc{\fy}{\frak{y}}\nc{\fz}{\frak{z}}

\newtheorem{theorem}{Theorem}[section]
\newtheorem{lemma}[theorem]{Lemma}
\newtheorem{corollary}[theorem]{Corollary}
\newtheorem{prop}[theorem]{Proposition}

\theoremstyle{definition}
\newtheorem{definition}[theorem]{Definition}
\newtheorem{example}[theorem]{Example}
\newtheorem{remark}[theorem]{Remark}
\newtheorem{eq:question}[theorem]{eq:question}
\newtheorem{sec:problem}[theorem]{sec:problem}
\newtheorem{conj}[theorem]{Conjecture}

\DeclareMathOperator{\rank}{rank}

 \DeclareMathOperator{\supp}{supp}
 
\DeclareMathOperator{\Hom}{{Hom}} 

\DeclareMathOperator{\sHom}{{\mathscr{H}om}}

 \DeclareMathOperator{\Lie}{Lie}
\DeclareMathOperator{\Spec}{{Spec}} 

 \DeclareMathOperator{\End}{End}

\DeclareMathOperator{\SL}{SL}

\newcommand{\cR}{\mathcal {R}}

\DeclareMathOperator{\Ind}{Ind}

\def\angl#1{{\langle #1\rangle}}

\DeclareMathOperator{\pt}{pt}

\newcommand{\Z}{\bbZ}

\newcommand{\La}{\Lambda}
\newcommand{\la}{\lambda}
\newcommand{\al}{\alpha}
\newcommand{\be}{\beta}

\newcommand{\de}{\delta}

\DeclareMathOperator{\St}{St}

\newcommand{\Dem}{\Delta}   
\newcommand{\DF}{\mathbb{D}}  
\newcommand{\DFd}{\mathbb{D}^*}
\newcommand{\tilx}{\tilde{x}}
\newcommand{\hatx}{\hat{x}}

\newcommand{\hatY}{\hat{Y}}
\newcommand{\hatQ}{\hat{Q}}
\newcommand{\unit}{\mathbf{1}}  
\DeclareMathOperator{\ev}{ev}

\newcommand{\taup}{{\overset{+}\tau}}
\newcommand{\taum}{{\overset{-}\tau}}
\newcommand{\taupm}{{\overset{\pm}\tau}}
\newcommand{\taump}{{\overset{\mp}\tau}}
\newcommand{\Stp}{\St^+}

\newcommand{\Stpm}{\St^\pm}

\newcommand{\Stm}{\St^-}
\newcommand{\ap}{a^+}
\newcommand{\bp}{b^+}
\newcommand{\am}{a^-}
\newcommand{\bm}{b^-}
\newcommand{\apm}{a^\pm}
\newcommand{\amp}{a^\mp}
\newcommand{\bpm}{b^\pm}
\newcommand{\bmp}{b^\mp}

\DeclareMathOperator{\stab}{stab}

\DeclareMathOperator{\Leaf}{Leaf}
\DeclareMathOperator{\Slope}{Slope}

\DeclareMathOperator{\Frac}{Frac}
\DeclareMathOperator{\Pic}{Pic}

\newcount\cols
{\catcode`,=\active\catcode`|=\active
 \gdef\Young(#1){\hbox{$\vcenter
 {\mathcode`,="8000\mathcode`|="8000
  \def,{\global\advance\cols by 1 &}%
  \def|{\cr
        \multispan{\the\cols}\hrulefill\cr
        &\global\cols=2 }%
  \offinterlineskip\everycr{}\tabskip=0pt
  \dimen0=\ht\strutbox \advance\dimen0 by \dp\strutbox
  \halign
   {\vrule height \ht\strutbox depth \dp\strutbox##
    &&\hbox to \dimen0{\hss$##$\hss}\vrule\cr
    \noalign{\hrule}&\global\cols=2 #1\crcr
    \multispan{\the\cols}\hrulefill\cr%
   }
 }$}}
}

\title[K-theoretic stable bases]
{On the K-theory stable bases of the Springer resolution}
\date{\today}

\author[C.~Su]{Changjian~Su}
\address{IHES, 35 Route de Chartres, 91440 Bures-sur-Yvette, France}
\email{changjiansu@gmail.com}

\author[G.~Zhao]{Gufang~Zhao}
\address{University of Massachusetts Amherst, Department of Mathematics and Statistics,  Amherst MA 01003}
\curraddr{Institute of Science and Technology Austria,
Am Campus, 1,
Klosterneuburg 3400,
Austria}
\email{gufang.zhao@ist.ac.at}

\author[C.~Zhong]{Changlong~Zhong}
\address{State University of New York at Albany, 1400 Washington Ave, ES 110, Albany, NY, 12222}
\email{czhong@albany.edu}

\begin{document}

\begin{abstract}
Cohomological and K-theoretic stable bases originated from the study of  quantum cohomology and quantum K-theory. Restriction formula for cohomological stable bases played an important role in computing the quantum connection of cotangent bundle of partial flag varieties. In this paper we study the K-theoretic stable bases of cotangent bundles of flag varieties.
We describe these bases in terms of the action of the affine Hecke algebra and the twisted group algebra of Kostant-Kumar. Using this algebraic description and the  method of root polynomials, we give a restriction formula of the stable bases. We apply it to obtain the  restriction formula for partial flag varieties. We also build a relation between the stable basis and the Casselman basis in the principal series representations of the Langlands dual group. As an application, we give a closed formula for the transition matrix between Casselman basis and the characteristic functions. 
\end{abstract}
\maketitle
\tableofcontents

\section{Introduction}
In \cite{MO12}, Maulik-Okounkov defined  the cohomological stable envelope for symplectic resolutions (see also \cite{BMO11}). 
The image of certain cohomology classes under the stable envelope are called the cohomological stable bases. The stable envelope is used to construct a quantum group action on the cohomology of quiver varieties, and to compute its quantum connection. Moreover, Nakajima gave a sheaf theoretic definition of the stable envelope \cite{N13}. 
We refer the readers to  \cite{BFN12,MS13,S13,S14} for other applications. 

The K-theoretic stable envelope is defined in \cite{MO}  (see also \cite{OK15, OS16, RTV15}). It is constructed in \cite{MO} and used to define a quantum group action  on the equivariant K-theory of quiver varieties \cite{OS16}. Based on that, in \cite{OK15}, difference equations in quantum K-theory of quiver varieties are constructed  geometrically,  which are further identified algebraically with the quantum Knizhnik-Zamolodchikov equations \cite{FR92, OK15} and quantum Weyl group actions \cite{OS16}. The monodromy of these difference equations is studied in \cite{AO} using the elliptic stable envelope. The K-theoretic stable bases for Hilbert scheme of points on $\bbC^2$ is studied in \cite{N15} and \cite{GN15}.

Stable basis for cotangent bundle of flag varieties and partial flag varieties are also of interest. The cohomological stable bases for $T^*(G/B)$ turns out to be the characteristic cycles of certain D-modules on the flag variety $G/B$. Pulling it back to $G/B$, we get the Chern--Schwartz--MacPherson classes for the Schubert cells \cite{AMSS17,RV15}. 
Moreover,  for cohomological stable bases of the cotangent bundle $T^*(G/P)$, in \cite{Su15}, the first-named author obtained their restriction formula, which  played an important role in computing the quantum connection of $T^*(G/P)$ in \cite{Su16}.

The goal of the present paper is to study the K-theory stable bases of cotangent bundle of flag varieties, and to find a restriction formula for the K-theoretic stable bases, formula expressing the stable bases in terms of the torus fixed point basis in $T^*(G/B)$. 
For each choice of a Weyl chamber, there is a set of stable basis, labeled by Weyl group elements $w\in W$.  For the positive/negative Weyl chambers, the stable basis will be denoted by  $\{\stab_{\pm}(w)\mid w\in W\}$. (There are other choices involved in the definition. See \S~\ref{sec:act1} for the detail.) In the special cases when $w\in W$ is the identity  $e$ or and the longest element $w_0\in W$, $\stab_+(e)$ and $\stab_-(w_0)$ are equal to the structure sheaves of the corresponding fixed points, up to a factor.
 
Let $Z$ be the Steinberg variety and $A$ be the maximal torus of $G$. The convolution algebra $K_{G\times \bbC^*}(Z)$, which is isomorphic to the affine Hecke algebra by a well known theorem of Kazhdan--Lusztig and Ginzburg (\cite{KL87, CG97}), acts on  $K_{A\times \bbC^*}(T^*G/B)$ on the left  and on the right.  Under these two actions, the Demazure-Lusztig operator corresponding to simple root $\alpha$ are denoted  respectively by  $T_\al$ and $T_\al'$. Our first main result is the following:

\begin{theorem}[Theorem \ref{thm:stabpim}] \label{thm:intro1} The elements $\stab_{\pm}(w)$  are generated by the action of $K_{G\times \bbC^*}(Z)$. More precisely, \[\stab_{+}(w)=q^{-\ell(w)/2}T_{w^{-1}}'(\stab_{+}(e)), \quad \stab_{-}(w)=q^{\ell(w_0w)/2}(T_{w_0w})^{-1}(\stab_{-}(w_0)).\]
\end{theorem}
In the proof of this theorem, we use the {\it rigidity} technique (see \S~\ref{sec:rigidity}) to calculate the affine Hecke algebra actions on the stable bases in Proposition \ref{prop:stabim}.

Theorem \ref{thm:intro1} allows us to give a purely algebraic definition of the stable bases (Definition~\ref{def:stabalg}), involving only the affine Hecke algebra, the twisted group algebra of   Kostant-Kumar and its dual. The study of properties of the stable bases boils down to combinatorics of the twisted group algebra.

We use Theorem \ref{thm:intro1} and the root polynomial method to find a restriction formula of stable bases. 
Such polynomials for cohomology and K-theory of flag varieties were  studied by Billey, Graham, and Willems \cite{B99, Gr02, W02}, and then generalized by Lenart-Zainoulline \cite{LZ14}. In this method,  a formula of the Schubert classes in terms of classes of torus fixed points are determined by the coefficients of  root polynomials  (see Theorem \ref{thm:rootcoeff}). Generalizing the root polynomial method, we obtain our second main result. For the cotangent bundle of partial flag varieties in type A, this is also obtained by Rim{\'a}nyi, Tarasov and Varchenko using weight functions in \cite{RTV15, RTV17}. In a work in progress of Knutson--Zinn-Justin, K-theory stable basis is also studied from the point of view of integrable systems. 

\begin{theorem}[Theorem \ref{thm:1}] \label{thm:intro2} With $a^+_{w,v}$ (resp. $K^\tau_{w,v}$) defined in Lemma \ref{lem:basis} (resp. \S\ref{subsec:rootDL}), we have
\begin{eqnarray*}
\stab_{+}(w)|_{v}&=&q^{-\ell(w)/2}v(a^+_{w^{-1}, v^{-1}})\prod_{\al>0}(1-e^\al).\\
\stab_{-}(w)|_v&=&q^{\ell(w)/2}K^\tau_{w,v}[\prod_{\al>0, v^{-1}\al>0}(1-qe^{-\al})]\cdot [\prod_{\al>0, v^{-1}\al<0}(1-e^\al)].
\end{eqnarray*}
\end{theorem}

We also give some applications of the above theorems in \S~\ref{sec:geoPJ}. 
We obtain the restriction formula for stable bases in $K_T(T^*G/P_J)$ in Theorem \ref{thm:resPJ}.
This is done by  showing that the stable bases coincide with the image of $\stab_{\pm}(w)\in K_T(T^*G/B)$ via the Lagrange correspondence from $T^*G/B$ to $T^*G/P_J$.

As an application, we study the relation between $K$-theory of the Springer resolution and the principal series representations of $p$-adic groups.

 In Theorem \ref{thm:comparison}, we relate the $T$-equivariant $K$-theory of the Springer resolution to the bases in the Iwahori invariants of an unramified principle series \cite{C80,R92}. Such an isomorphism has been well-known, and has been studied by Lusztig \cite{Lus98} and Braverman--Kazhdan \cite{BK99} from different points of view.  
 However, the present paper explicitly identity different bases from $K$-theory and from $p$-adic  representation theory, which had been previously unknown. In particular, the $K$-theory stable basis is identified with the characteristic functions on certain semi-infinite orbits; the $T$-fixed-point basis is identified with the Casselman basis. 
Consequently, Theorem~\ref{thm:intro2} also gives a closed formula for the transition matrix between these characteristic functions and the Casselman basis. A formula for the generating function of the matrix coefficients has been previously achieved by Reeder via a different approach \cite[Proposition 5.2]{R92}.

Under the isomorphism in Theorem \ref{thm:comparison}, various structures from the $p$-adic representations, e.g., the intertwiners, Macdonald's formula for the spherical functions \cite{M68,C80}, and the Casselman--Shalika formula for Whittaker functions \cite{CS80}, have meanings in terms of $K$-theory. Although this isomorphism is well-known, the $K$-theory interpretation of these structures are not well-documented. 
For the convenience of the readers, we also spell these out in \S~\ref{sec:padic}.

The results in the present paper also provide a way to study the transition matrix between stable bases and the Schubert classes of $K_T(G/B)$, as will be explained in a future publication. 
Such transition matrix is related with  \cite{LLL16} which studies the (spherical) Whittaker functions of $p$-adic groups. It is also shadowed by the two geometric realizations of the affine Hecke algebras \cite{B16} and the periodic modules \cite{BK99, Lus97, Lus98}.  
The cohomological analogue of this transition matrix, i.e., the transition matrix from cohomological Schubert classes to the cohomological stable bases, is of independent interest. 
It was proved in \cite{AMSS17} that cohomological stable bases can be identified with Chern-Schwartz-MacPherson classes. In \cite{AM15}, Aluffi and Mihalcea raised a positivity conjecture concerning this matrix. 
Recently, the non-equivariant case is proved in \cite{AMSS17}, in which the cohomological stable basis played an important role.

Another future application is a relation between the $K$-theory stable basis and the localizations of baby Verma modules in modular representations of Lie algebras, as will be explained in a separate paper. 

The structure of this paper is as follows: 
In Section \ref{sec:geostable} we recall the definition of stable bases. In Section \ref{sec:rigidity} we recall rigidity in K-theory. In Section \ref{sec:Heckeaction} we define the two convolution actions by the Hecke algebra, and compute their effects on the stable bases.   
In Section \ref{sec:affHecke} we recall an algebraic description of affine Hecke algebra in terms of Kostant-Kumar's twisted group algebra. In Section \ref{sec:algstable} we give an algebraic description of stable bases. In Section \ref{sec:root} we define the root polynomials for Hecke algebra and in Theorem \ref{thm:rootcoeff}; we relate some coefficients of Hecke algebra with root polynomials, and obtain the restriction formula in Theorem \ref{thm:1}. In Section \ref{sec:geoPJ} we give an algebraic description of the stable bases for partial flag varieties. In Section~\ref{sec:padic} we talk about the relation between stable basis and the Casselman basis in $p$-adic representations.

\subsection*{Acknowledgment}
The first named author is grateful to his advisor A. Okounkov for teaching him the geometrical definition of the K-theoretic stable envelope. We would also like to thank E. Gorsky, C. Lenart,  A. Negu{\c{t}}, A. Smirnov and Z. Zhou for lots of helpful discussions, and to C. Lenart for pointing out to us of the paper \cite{NN15}. We thank J. Schuermann for pointing out a mistake in a previous version of this paper. The third author  was partially supported by the Simons Foundation and by the Mathematisches Forschungsinstitut Oberwolfach (MFO) under the program of Simons Visiting Professorship, and also supported by  Marc Levine during a research stay at the University of Duisburg-Essen. 

\section*{Notations}
Through out this paper,  $G$ is a complex reductive  group with maximal torus $A$, a Borel subgroup $B$ and its opposite Borel subgroup $B^-$. 
Let $\La$ be the group of characters of $A$.
Let $\Sigma$ be the set of roots of $G$. Let $\Sigma^+$ be the roots in $B$, which is the set of positive roots, and let $\Sigma^-$ be the negative roots. For each root $\al$, we use $\al>0$ or $\al<0$ to say that it is positive or negative. Let $\Pi=\{\al_1,...,\al_n\}$ be the set of simple roots, and  $\rho=\frac12\sum_{\al\in \Sigma^+}\al$. Let $\geq$ denote the Bruhat order in the Weyl group $W$.

Let $G/B$ be the complete flag variety.  The maximal torus $A$ acts on $G/B$ by left multiplication. Hence, it also the cotangent bundle $T^*G/B$ and the tangent bundle $T(G/B)$ are equivariant. 
Let $T=A\times \bbC^*$.
We denote the standard representation of $\bbC^*$ by $q^{\frac{1}{2}}$.  
The factor $\bbC^*\subseteq T$ acts trivially on $G/B$, and dilates the fibers of  $T^*G/B$ by the character $q^{-1}$.  The $T$-fixed points of $T^*G/B$ and $G/B$ are both bijective to $W$, the Weyl group of $G$.

For any $J\subset \Pi$, let $W_J\subset W$ be the corresponding subgroup, $W^J$ be the set of minimal length representatives, and $G/P_J$ be the corresponding variety of partial flags.  Let $\Sigma_J=\{\al\in \Sigma|s_\al\in W_J\}$, and similarly define $\Sigma_J^\pm$.  Let $w_0$ be the longest element of $W$, and $w_0^J$ the longest element of $W_J$.  For a reduced decomposition $w=s_{i_1}\cdots s_{i_l}$, define
\[
\Sigma_w:=w\Sigma^-\cap \Sigma^+=\{\al_{i_1}, s_{i_1}(\al_{i_2}), ..., s_{i_1}s_{i_2}\cdots s_{i_{l-1}}(\al_{i_l})\}.
\]  We will frequently use the identities\[w_0\Sigma^-=\Sigma^+, \quad s_i\Sigma^-=(\{\al_i\}\sqcup\Sigma^-)\backslash\{-\al_i\}, \quad v(\Sigma^+\backslash \Sigma^+_J)=\Sigma^+\backslash \Sigma^+_J, \quad \text{for }v\in W_J.\]

Let $R=\Z[q^{\frac{1}{2}}, q^{-\frac{1}{2}}], S=R[\La]$, then $S\cong K_T(\bbC)$, and let $Q=\Frac (S)$ be its field of fractions.

\section{Stable bases of $T^*G/B$}\label{sec:geostable}
In this section, we recall Maulik and Okounkov's  definition of the K-theoretic stable bases for the Springer resolution.  

Recall that $\bbC^*$ act on the cotangent fiber of $T^*G/B$ by a non-trivial character $q^{-1}$, where $q^{-\frac12}$ corresponds to the standard representation of the torus $\bbC^*$. Therefore $K_{\bbC^*}(\bbC)=R=\Z[q^{\frac12},q^{-\frac12}], K_T(\bbC)\cong S=R[\La]$. For any $T$-invariant vector space $V$, let 
\[\bigwedge\nolimits^\bullet V:=\sum_k(-1)^k\bigwedge\nolimits^k V^\vee=\prod (1-e^{-\alpha})\in K_T(\bbC),\]
where the product is over all $\Lie(T)$-weights in $V$ counted with multiplicities.  

The $T$-fixed loci of $T^*G/B$ is discrete and is in one-to-one correspondence with $W$. For an element $w\in W$,  the corresponding fixed point is still denoted by $w$. Let $\iota_{w}$ be the inclusion of the fixed point $w\in W$ into $T^*G/B$. By Thomason's theorem \cite{Th}, 
$K_T(T^*G/B)\otimes_{S}Q$ is a finite dimensional $Q$-vector space with basis $\{\iota_{w*}1|w\in W\}$. This basis is referred to as the fixed-point basis. For any $\calF\in K_T(T^*G/B)$, let $\calF|_w$ denote the restriction of $\calF$ to the fixed point $w\in T^*G/B$. Let $(\cdot,\cdot)$ denote the K-theoretic pairing on $K_T(T^*G/B)$, which can be  defined using localization as follows:
\[(\calF_1,\calF_2)=\sum_w\frac{\calF_1|_w\otimes\calF_2|_w}{\prod_{\alpha>0}(1-e^{w\alpha})(1-qe^{-w\alpha})},\quad \calF_1, \calF_2\in K_T(T^*G/B).\]

\subsection{The definition of stable bases}\label{sec:geodef}

Let $\La^\vee$ be the lattice of cocharacters of $A$. The Lie algebra of the maximal compact subgroup of $A$ is $\mathfrak{a}_{\mathbb{R}}=\La^\vee\otimes_{\mathbb{Z}}\mathbb{R}.$
The $A$-weights occurring in the normal bundle  $(T^*G/B)^A$ coincide with the usual roots for $G$. The root hyperplanes $\alpha_i^\perp$ partition $\mathfrak{a}_{\mathbb{R}}$ into finitely many chambers
\[
\mathfrak{a}_{\mathbb{R}}\setminus\bigcup \alpha_i^\perp=\coprod \mathfrak{C}_i.
\]
 Let $+$ denote the chamber such that all roots in $\Sigma^+$ are positive on it, and $-$ the opposite chamber. 
Let $\mathfrak{C}$ be a chamber. For any cocharacter  $\sigma\in \mathfrak{C}$,  the stable leaf of $w\in W$ is defined as
\[
\Leaf_{\mathfrak{C}}(w)=\left\{x\in T^*G/B \mid \lim\limits_{z\rightarrow 0}\sigma(z)\cdot x=w . \right\}.
\]
Note that the limit, and hence $\Leaf_{\mathfrak{C}}(w)$ per se,  is independent of the choice of $\sigma$. In particular,  $\Leaf_+(w)=T_{BwB/B}^*G/B$, and $\Leaf_-(w)=T_{B^-wB/B}^*G/B$. Define a partial order on W as follows:
\[
w\preceq_{\mathfrak{C}} v \text{\quad if \quad} \overline{\Leaf_{\mathfrak{C}}(v)}\cap w\neq \emptyset.
\]
Then the order $\preceq_+$ determined by the positive chamber is the same as the Bruhat order $\leq$, and $\preceq_-$ is the opposite Bruhat  order.  Define the slope of a fixed point $v$ by
\[
\Slope_{\mathfrak{C}}(v)=\bigcup_{w\preceq_{\mathfrak{C}} v} \Leaf_{\mathfrak{C}}(w).
\] 

\begin{definition}
A {\it polarization} $T^{\frac{1}{2}}\in K_T(T^*G/B)$ is the choice of a Lagrangian subbundle of the tangent bundle $T(T^*G/B)\in K_T(T^*G/B)$, so that
\[T^{\frac{1}{2}}+q^{-1}(T^{\frac{1}{2}})^\vee=T(T^*G/B)\]  as $T$-equivariant vector bundles.
\end{definition}
For any polarization $T^{\frac{1}{2}}$, there is an opposite one defined as
\[T^{\frac{1}{2}}_{\text{opp}}=q^{-1}(T^{\frac{1}{2}})^\vee.\]
There are two natural polarizations: $T(G/B)$ and $T^*G/B$ which are opposite to each other. Let $N_w$ denote the normal bundle of $T^*G/B$ at $w\in W$. 

Any chamber $\mathfrak{C}$ determines a decomposition $N_w=N_{w,+}\oplus N_{w,-}$   into $A$-weight spaces which are positive and negative with respect to $\mathfrak{C}$  respectively. For any polarization $T^{\frac{1}{2}}$,  denote $N_w^{\frac{1}{2}}$ by $N_w\cap T^{1/2}|_w$. Similarly, we have $N_{w,+}^{\frac{1}{2}}$ and $N_{w,-}^{\frac{1}{2}}$. In particular,   $N_{w,-}=N^{\frac{1}{2}}_{w,-}\oplus q^{-1}(N_{w,+}^{\frac{1}{2}})^\vee$. Consequently, we have
\[N_{w,-}- N_w^{\frac{1}{2}}=q^{-1}(N_{w,+}^{\frac{1}{2}})^\vee- N_{w,+}^{\frac{1}{2}}\] as virtual vector bundles, 
whose determinant is a complete square. Therefore,  we denoted by
\[\left(\frac{\det N_{w,-}}{\det N_w^{\frac{1}{2}}}\right)^{\frac{1}{2}}\]
its square root. 

Recall that for any  weight $\lambda$,  let $\calL_\lambda$ be the associated line bundle on $G/B$.  Pulling it back to $T^*G/B$ via the projection map, we get the corresponding line bundle on $T^*G/B$, denoted by $\calO(\lambda)$. The assignment associating $\lambda\in \La$  to $\calO(\lambda)\in \Pic_A(T^*G/B)$ induced an isomorphsim.  For every rational weight $\lambda\in P\otimes_\bbZ \bbQ$, let $\calO(\lambda)$ denote the corresponding element in $\Pic_A(T^*G/B)\otimes_\bbZ \bbQ$.  We say $\lambda$, or the corresponding $\calO(\la)$ is sufficiently general if 
\begin{equation}\label{eq:condLineBun}
\lambda-w\lambda\notin \La \text{ for any } w \in W. 
\end{equation}

For a Laurent polynomial $f:=\sum_\mu f_\mu z^\mu\in K_T(\text{pt})$, where $e^\mu\in K_A(\text{pt})$ and $f_\mu\in R$, we define its \textit{Newton Polygon}, denoted by $\deg_Af$ to be
\[\deg_A f=\text{Convex hull} (\{\mu| f_\mu\neq 0\})\subset \La\otimes_\bbZ \bbQ.\]

We use the following theorem  as the definition of K-theoretic stable bases.
\begin{theorem}\label{thm:geostable} \cite[\S 9.1]{OK15}
For any chamber $\fC$, a sufficiently general $\calL$, and a polarization $T^{1/2}$, there exists a unique map of $S$-modules
\[
\stab_{\mathfrak{C},T^{\frac{1}{2}},\calL}:K_T((T^*G/B)^A)\rightarrow K_T(T^*G/B)
\]
such that for any $w\in W$, $\Gamma=\stab_{\mathfrak{C},T^{\frac{1}{2}},\calL}(w)$ satisfies:
\begin{enumerate}
\item (\textit{support}) $\supp \Gamma\subset \Slope_{\mathfrak{C}}(w)$;
\item (\textit{normalization}) $\Gamma|_w=(-1)^{\rank N_{w,+}^{\frac{1}{2}}}\left(\frac{\det N_{w,-}}{\det N_w^{\frac{1}{2}}}\right)^{\frac{1}{2}}\calO_{\Leaf_\mathfrak{C}(w)}|_w$;
\item (\textit{degree}) $\deg_A\left(\Gamma|_v\otimes \calL|_w\right)\subseteq \deg_A\left((\stab_{\mathfrak{C},T^{\frac{1}{2}},\calL}(v)\otimes\calL)|_v\right)$, for any $v\prec_{\mathfrak{C}} w$,
\end{enumerate}
where $w$ in $\stab_{\mathfrak{C},T^{\frac{1}{2}},\calL}(w)$ is the unit in $K_T^*(w)$.
\end{theorem}
\subsection{Comments and examples}
\begin{remark}\label{rem:afterdef} 
(1). 
From the characterization, the transition matrix from $\{\stab_{\mathfrak{C},T^{\frac{1}{2}},\calL}(w)| w\in W\}$ to the fixed point basis is a triangular matrix with nonzero diagonal terms. Hence, after localization, $\{\stab_{\mathfrak{C},T^{\frac{1}{2}},\calL}(w)| w\in W\}$ form a basis, which is called the \textit{stable bases}.

(2).
It is shown in \cite[Proposition 1]{OS16} that via the K-theory pairing, $\{\stab_{\mathfrak{C},T^{\frac{1}{2}},\calL}(w)| w\in W\}$ and $\{\stab_{-\mathfrak{C},T_{\text{opp}}^{\frac{1}{2}},\calL^{-1}}(w)| w\in W\}$ are dual to each other, i.e.,
\[\left(\stab_{\mathfrak{C},T^{\frac{1}{2}},\calL}(v),~\stab_{\mathfrak{-C},T_{\text{opp}}^{\frac{1}{2}},\calL^{-1}}(w)\right)=\delta_{v,w}.\]

(3).
Let the alcoves of $\mathfrak{g}=\Lie G$  to be the connected components of $(\Lie A_\bbR)^*\setminus H_{\alpha,n}$, with $(Lie A_\bbR)^*=\La$ and  $H_{\alpha,n}=\{\lambda\in (\Lie A_\bbR)^*|(\lambda, \alpha^\vee)=n\}$. Then $\stab_{\mathfrak{C},T^{\frac{1}{2}},\calL}$ stays the same if $\calL$ is in the same alcove. I.e., $\stab_{\mathfrak{C},T^{\frac{1}{2}},\calL}$ depends  on $\calL$ locally.

(4). 
By the uniqueness property, we have
\[\stab_{\mathfrak{C},T^{\frac{1}{2}},\calL\otimes \calO(\lambda)}(w)=e^{-w\lambda}[\calO(\lambda)]\otimes\stab_{\mathfrak{C},T^{\frac{1}{2}},\calL}(w),\]
where $\lambda$ is an integral weight of $\mathfrak{g}$. Combining with part (3), it is sufficient to study stable bases for those alcoves near $0\in \Lie A^*_\bbR$.

(5). 
Let us explain why Condition~\eqref{eq:condLineBun} on $\calL=\calO(\lambda)$ is imposed. Suppose $\mu:=v\lambda-u\lambda\in \La$ for some $v\prec_{\mathfrak{C}} u$, and suppose we already have a map $\stab_{\mathfrak{C},T^{\frac{1}{2}},\calL}$ as in Theorem \ref{thm:geostable}. For any Laurent polynomial $f(q)$, we define a new map $\stab'_{\mathfrak{C},T^{\frac{1}{2}},\calL}$ as follows:
\[\stab'_{\mathfrak{C},T^{1/2},\calL}(y)=
\left\{ \begin{array}{cc}
\stab_{\mathfrak{C},T^{1/2},\calL}(y),& \text{ if }  y\neq u;\\
\stab_{\mathfrak{C},T^{1/2},\calL}(u)+f(q)e^{\mu}\stab_{\mathfrak{C},T^{1/2}, \calL}(v),& \text{ if } y=u.
\end{array}\right. \]
 Then the new map also satisfies the conditions in Theorem \ref{thm:geostable}, which contradicts with the uniqueness. We check this as follows: The first two conditions are obvious. For the degree condition, there are four cases: $u\prec_{\mathfrak{C}}y$, $y=u$, $y\prec_{\mathfrak{C}}u$, and $y$ is not comparable with $u$. The last two cases are easy to check, so we only consider the first two. 
\begin{itemize}
\item Case $u\prec_{\mathfrak{C}}y$.\\
In this case, $\stab'_{\mathfrak{C},T^{1/2},\calL}(y)=\stab_{\mathfrak{C},T^{1/2},\calL}(y)$. If $w\neq u$, then 
$\stab'_{\mathfrak{C},T^{1/2},\calL}(w)=\stab_{\mathfrak{C},T^{1/2},\calL}(w)$. Thus the condition is satisfied. If $w=u$, then 
$\stab'_{\mathfrak{C},T^{1/2},\calL}(u)|_u=\stab_{\mathfrak{C},T^{1/2},\calL}(u)|_u$ because $v\prec_{\mathfrak{C}}u$. So the condition is also satisfied.
\item Case $y=u$.\\
In this case, $\stab'_{\mathfrak{C},T^{1/2},\calL}(w)=\stab_{\mathfrak{C},T^{1/2},\calL}(w)$, for any $w\prec_{\mathfrak{C}}u$.
By definition,
\[\deg_A\left(\stab'_{\mathfrak{C},T^{1/2},\calL}(u)|_w\right)+u\lambda=\deg_A\left(\stab_{\mathfrak{C},T^{1/2},\calL}(u)|_w+f(q)e^\mu\stab_{\mathfrak{C},T^{1/2},\calL}(v)|_w\right)+u\lambda.\]
Since 
\[\deg_A\left(\stab_{\mathfrak{C},T^{1/2},\calL}(u)|_w\right)+u\lambda\subset \deg_A\left(\stab_{\mathfrak{C},T^{1/2},\calL}(w)|_w\right)+w\lambda\]
and 
\[\deg_A\left(f(q)e^\mu\stab_{\mathfrak{C},T^{1/2},\calL}(v)|_w\right)+u\lambda\subset \deg_A\left(\stab_{\mathfrak{C},T^{1/2},\calL}(w)|_w\right)+w\lambda,\] 
therefore,  \[\deg_A\left(\stab'_{\mathfrak{C},T^{1/2},\calL}(u)|_w\right)+u\lambda\subset \deg_A\left(\stab'_{\mathfrak{C},T^{1/2},\calL}(w)|_w\right)+w\lambda.\]
\end{itemize}

\end{remark}

\begin{example}
Let us study the easiest example in which $G=\SL(2,\bbC)$, and hence $G/B=\bbP^1$. Let $\alpha$ be the unique positive root. Let $0$ and $\infty$ denote the two fixed points, which correspond to $1$ and $s_\alpha$ in the Weyl group. Then $S=\bbZ[q^{\pm\frac{1}{2}}][e^{\pm\frac{\alpha}{2}}]$,  $T_0\bbP^1$ has weight $e^{-\alpha}$, $T_0^*\bbP^1$ has weight $q^{-1}e^{\alpha}$, $T_\infty\bbP^1$ has weight $e^{\alpha}$, and $T_\infty^*\bbP^1$ has weight $q^{-1}e^{-\alpha}$.  The condition (\ref{eq:condLineBun}) on $\lambda$ is equivalent to $\lambda\notin \frac{\bbZ}{4}\alpha$. The alcoves are $(\frac{n\alpha}{2},\frac{(n+1)\alpha}{2})$, where $n\in \bbZ$. 

Let us pick the negative chamber, and fix the polarization to be $T^*\bbP^1$. Then $\Leaf(\infty)=T_\infty^*\bbP^1$, and $\Leaf(0)=\bbP^1\setminus \{\infty\}$. Thus $0>\infty$. It is easy to see that for any slope 
$\calO(\lambda)$, 
\[\stab(\infty)=-q^{\frac{1}{2}}e^\alpha[\calO_{T_\infty^*\bbP^1}].\]

By  Remark \ref{rem:afterdef}.(4), we only need to study the case for a fixed $\lambda\in (0,\frac{\alpha}{2})$. By the support and normalization conditions in Theorem \ref{thm:geostable}, we get
\[\stab(0)=[\calO_{\bbP^1}]+a[\calO_{T_\infty^*\bbP^1}], \quad a\in Q.\] However, by the support condition, $(\stab(0),[\calO_{\bbP^1}])\in S=\bbZ[e^{\pm\frac{\alpha}{2}}][q^{\frac{1}{2}},q^{-\frac{1}{2}}]$. Since $([\calO_{\bbP^1}], [\calO_{\bbP^1}])\in S$ and $([\calO_{T_\infty^*\bbP^1}], [\calO_{\bbP^1}])=1$, we get $a\in S$.

We have
\[\stab(0)|_\infty=1-qe^\alpha+a(1-e^{-\alpha}).\]
Since $\deg_A\stab(\infty)|_\infty=[0,\alpha]$, we have
\[\deg_A(1-qe^\alpha+a(1-e^{-\alpha}))\subset [s_\alpha\lambda-\lambda,s_\alpha\lambda-\lambda+\alpha]=[-(\lambda,\alpha^\vee)\alpha,-(\lambda,\alpha^\vee)\alpha+\alpha].\]
There are two cases.

\noindent(1). {\it  Case $\lambda\in (0,\frac{\alpha}{4})$:}

 In this case, 
\[\deg_A(1-qe^\alpha+a(1-e^{-\alpha}))\subset[-(\lambda,\alpha^\vee)\alpha,-(\lambda,\alpha^\vee)\alpha+\alpha]\subset (-\frac{\alpha}{2}, \alpha).\]
Since $a\in S$, we get $a=qe^\alpha$.

\noindent(2). {\it Case $\lambda\in (\frac{\alpha}{4},\frac{\alpha}{2})$:}

In this case, 
\[\deg_A(1-qe^\alpha+a(1-e^{-\alpha}))\subset[-(\lambda,\alpha^\vee)\alpha,-(\lambda,\alpha^\vee)\alpha+\alpha]\subset (-\alpha,\frac{\alpha}{2}).\]
Then $a$ must be of the form $a_1e^\alpha+a_2e^{\frac{\alpha}{2}}$ for some $a_i\in R$. Plugging into $1-qe^\alpha+a(1-e^{-\alpha})$, we get $1-qe^\alpha+a_1e^\alpha+a_2e^{\frac{\alpha}{2}}+a_1+a_2e^{-\frac{\alpha}{2}}$. Since $\deg_A(1-qe^\alpha+a_1e^\alpha+a_2e^{\frac{\alpha}{2}}+a_1+a_2e^{-\frac{\alpha}{2}})\subset (-\alpha,\frac{\alpha}{2})$, we get $a_1=q$, $a_2=0$. Thus $a=qe^\alpha$.

To conclude, when $\lambda_0\in (0,\frac{\alpha}{2})$, we have 
\[\stab_{-,T^*\bbP^1,\lambda_0}(0)=[\calO_{\bbP^1}]+qe^\alpha[\calO_{T_\infty^*\bbP^1}],\]
and 
\[\stab_{-,T^*\bbP^1,\lambda_0}(\infty)=-q^{\frac{1}{2}}e^\alpha[\calO_{T_\infty^*\bbP^1}].\]
In general, when $\lambda_n\in (\frac{n\alpha}{2},\frac{(n+1)\alpha}{2})$, then $\lambda-\frac{n}{2}\alpha\in (0,\frac{\alpha}{2})$. Thus for $w=1, s_\alpha$, 
\[\stab_{-,T^*\bbP^1,\lambda_n}(w)=e^{-\frac{n}{2}w\alpha}[\calO(\frac{n}{2}\alpha)]\otimes\stab(w).\]

For the positive chamber, the opposite polarization $T\bbP^1$ and the opposite slope $\lambda_{-1}\in (-\frac{\alpha}{2},0)$, we have
\[\stab_{+,T\bbP^1,\lambda_{-1}}(0)=[\calO_{T_0^*\bbP^1}],\]
and 
\[\stab_{+,T\bbP^1,\lambda_{-1}}(\infty)=-q^{-\frac{1}{2}}e^{-\alpha}[\calO_{\bbP^1}]+\left(-q^{\frac{1}{2}}e^{-2\alpha}+(q^{-\frac{1}{2}}-q^{\frac{1}{2}})e^{-\alpha}\right)[\calO_{T_0^*\bbP^1}].\]
It is easy to check that 
\[\left(\stab_{+,T\bbP^1,\lambda_{-1}}(v), \stab_{-,T^*\bbP^1,\lambda_0}(w)\right)=\delta_{v,w}.\]
\end{example}

\section{Rigidity}\label{sec:rigidity}
In this section, we introduce rigidity, and make the normalization axiom for stable bases more explicit. 

In equivariant cohomology, degree counting is a very useful method in computations. In equivariant K-theory, this  method is often replaced by a rigidity argument.  If a $T$-equivariant sheaf $\calF$ has compact support, then the equivariant holomorphic Euler characteristic $\chi(\calF)\in K_T(\pt)$ is a Laurent polynomial, which, in general, is difficult to calculate. However, the calculation is simplified if $\chi(\calF)$ depends on few or even no equivariant variables. This property is known as rigidity. One standard way to  prove such a property is
to use the following elementary observation: for any $p(z)\in \bbC[z^{\pm}]$,
\begin{equation}\label{eq:obser}p(z) \text{ is bounded as } z^{\pm 1}\rightarrow \infty \Longleftrightarrow p=\text{constant}.
\end{equation}
For applications of this observation, see \cite[\S 2.4]{OK15}.

We will need the following lemma.
\begin{lemma}\label{lem:rigidity}
Suppose $f=\sum\limits_{\mu\in I} a_\mu e^\mu\in S$ is a Laurent polynomial, with $\mu\in (\Lie A)^*$ and $0\neq a_\mu\in R$.
\begin{enumerate}
\item 
There exists a  $\xi\in \Lie A_\bbR$ in the positive chamber, such that for any $\mu\in I$, $(\xi,\mu)\in \bbZ$, and further more,  $(\xi,\mu)\neq (\xi,\mu')$ for any $\mu\neq \mu'$ in $I$.
\item 
Moreover, if both of the limits $\lim\limits_{t\rightarrow \pm\infty}f(t\xi)$ are bounded, and one of them equals $g(q)$ for some $g(q)\in R$, then $f=g(q)$.
\end{enumerate}
\end{lemma}
\begin{proof}
The existence of such $\xi$ follows easily from the fact that $f$ has only finitely many terms. The second part follows from \eqref{eq:obser}. \end{proof}
With Lemma \ref{lem:rigidity} we can define the following two scalars
\[max_\xi f=\max_{\mu\in I}(\mu, \xi) \text{\quad and \quad} min_\xi f=\min_{\mu\in I}(\mu, \xi).\]

For any $v\in W$, we denote $q^{\ell(v)}$ simply by $q_v$.

By Theorem \ref{thm:geostable}, we have the following. 
\begin{lemma}\label{lem:charstab}
For $v,w\in W$,
\begin{enumerate}
\item 
$\stab_{-,T^*G/B,\calL}(v)|_w=0$, unless $w\geq v$;
\item 
$\stab_{-,T^*G/B,\calL}(v)|_v=q_v^{\frac{1}{2}} \underset{\al\in \Sigma^-\cap v\Sigma^-}\prod(1-qe^{\al})\cdot \underset{{\al\in \Sigma^+\cap v\Sigma^-}}\prod(1-e^{\al})$;
\item 
$\stab_{+,T(G/B),\calL}(v)|_w=0$, unless $w\leq v$;
\item 
$\stab_{+,T(G/B),\calL}(v)|_v=q_v^{-\frac{1}{2}}\underset{\al\in \Sigma^-\cap v\Sigma^+}\prod(q-e^\alpha)\cdot \underset{\al\in \Sigma^+\cap v\Sigma^+}\prod(1-e^\alpha)$.
\end{enumerate}
\end{lemma}
\begin{proof}
(1) and (3) follow from the the support condition. 

Now we prove (2). For the negative chamber, we have
\begin{flalign*}A-\text{weights in } N_{v,+}&=\{e^{-v\beta}|\beta>0,v\beta>0\}\cup \{q^{-1}e^{v\beta}|\beta>0,v\beta<0\},\\
A-\text{weights in } N_{v,-}&=\{e^{-v\beta}|\beta>0,v\beta<0\}\cup \{q^{-1}e^{v\beta}|\beta>0,v\beta>0\},\\
A-\text{weights in } N_v^{\frac{1}{2}}&=\{q^{-1}e^{v\beta}|\beta>0\}.\end{flalign*}
Therefore,
\begin{flalign*}
\stab_{-,T^*G/B,\calL}(v)|_v&=(-1)^{\rank N_{v,+}^{\frac{1}{2}}}\left(\frac{\det N_{v,-}}{\det N_v^{\frac{1}{2}}}\right)^{\frac{1}{2}}\calO_{\Leaf_\mathfrak{C}(v)}|_v\\
&=(-1)^{\ell(v)}\left(\frac{\prod_{\beta>0,v\beta<0}e^{-v\beta}\prod_{\beta>0,v\beta>0}q^{-1}e^{v\beta}}{\prod_{\beta>0}q^{-1}e^{v\beta}}\right)^{\frac{1}{2}}\prod_{\beta>0,v\beta<0}(1-e^{v\beta})\prod_{\beta>0,v\beta>0}(1-qe^{-v\beta})\\
&\overset{\sharp_1}=(-1)^{\ell(v)}q_v^{\frac{1}{2}}\prod\limits_{\beta>0,v\beta<0}(e^{-v\beta}-1)\prod\limits_{\beta>0,v\beta>0}(1-qe^{-v\beta}).
\end{flalign*}
We comment on the proof of the equality $\sharp_1$. In $\left(\frac{\prod_{\beta>0,v\beta<0}e^{-v\beta}\prod_{\beta>0,v\beta>0}q^{-1}e^{v\beta}}{\prod_{\beta>0}q^{-1}e^{v\beta}}\right)^{\frac{1}{2}}$, the factors involving powers of $e$ can be regrouped into two copies of $\Sigma^+\cap v\Sigma^-.$ Equality $\sharp_1$ then follows from the identity $e^{-\al}(1-e^\al)=e^{-\al}-1$. 

(4) follows from a similar argument as that of (2).
\end{proof}
Lemma \ref{lem:charstab}  implies that $\left(\stab_{+,T(G/B),\calL}(v), \stab_{-,T^*G/B,\calL}(v)\right)=1$,  
keeping in mind that
\[\underset{{\al\in \Sigma^+\cap v\Sigma^-}}\prod(1-e^{\al})\cdot \underset{\al\in \Sigma^-\cap v\Sigma^-}\prod(1-qe^{\al})\underset{\al\in \Sigma^-\cap v\Sigma^+}\prod(q-e^\alpha)\cdot \underset{\al\in \Sigma^+\cap v\Sigma^+}\prod(1-e^\alpha)=\bigwedge\nolimits^\bullet T_v(T^*G/B).\]

Choosing $\xi\in\Lie A$ as in Lemma \ref{lem:rigidity}, regarding the Laurent polynomials $\stab_{+,T(G/B),\calL}(v)|_v$ and $\stab_{-,T^*G/B,\calL}(v)|_v$, we have
\begin{align}\label{deg con}
max_\xi(\stab_{+,T(G/B),\calL}(v)|_v)=(\xi, \sum_{\beta>0,v\beta>0}v\beta),&\quad min_\xi(\stab_{+,T(G/B),\calL}(v)|_v)=(\xi, \sum_{\beta>0,v\beta<0} v\beta),\\
max_\xi(\stab_{-,T^*G/B,\calL}(v)|_v)=(\xi, \sum_{\beta>0,v\beta<0} -v\beta),&\quad min_\xi(\stab_{-,T^*G/B,\calL}(v)|_v)=(\xi, \sum_{\beta>0,v\beta>0} -v\beta).
\end{align}
Let $\rho$ be half sum of all the positive roots. For  any simple root $\al$,  we have
\begin{gather}max_\xi(\stab_{+,T(G/B),\calL}(v)|_v)+max_\xi(\stab_{-,T^*G/B,\calL}(v)|_v)=(\xi,2\rho),\\
min_\xi(\stab_{+,T(G/B),\calL}(v)|_v)+min_\xi(\stab_{-,T^*G/B,\calL}(v)|_v)=-(\xi,2\rho),\\
\label{eq:degalpha} max_\xi(\stab_{+,T(G/B),\calL}(vs_\alpha)|_{vs_\alpha})+max_\xi(\stab_{-,T^*G/B,\calL}(v)|_v)+(\xi,v\alpha)=(\xi,2\rho),\\
\label{eq:degalpha2}
min_\xi(\stab_{+,T(G/B),\calL}(vs_\alpha)|_{vs_\alpha})+min_\xi(\stab_{-,T^*G/B,\calL}(v)|_v)+(\xi,v\alpha)=-(\xi,2\rho).\end{gather}

\section{The two Hecke actions}\label{sec:Heckeaction}
In this section, we compute the action of the affine Hecke algebra on stable bases.
\subsection{Reminder on the Demazure-Lusztig operators}
Let $Z=T^*G/B\times_\calN T^*G/B$ be the Steinberg variety, where $\calN$ is the nilpotent cone. Let $\bbH$ be the affine Hecke algebra (see Chapter 7 in \cite{CG97}). There is an isomorphism 
\begin{equation}
\label{eqn:KL}
\bbH\simeq K_{G\times \bbC^*}(Z)\end{equation} defined as follows. The diagonal $G$-orbits on $G/B\times G/B$ are indexed by the Weyl group. For each simple root $\alpha\in \Pi$, let $Y^\circ_{\alpha}$ be the orbit corresponding to the simple reflection $s_{\alpha}$, whose closure is
\[Y_{\alpha}:=\overline{Y^\circ_{\alpha}}=G/B\times_{\mathcal{P}_{\alpha}}G/B,\]
where $\mathcal{P}_{\alpha}=G/P_{\alpha}$ and $P_{\alpha}$ is the minimal parabolic subgroup corresponding to  $\alpha$. Therefore, only two kinds of torus fixed points lie in $Y_\alpha$: $(w,w)$ and $(w,ws_\alpha)$. Let $\Omega_{\alpha}$ be the sheaf of differentials along the first projection from $Y_\alpha$ to $G/B$. Let $T_{Y_{\alpha}}^*:=T_{Y_{\alpha}}^*(G/B\times G/B)$ be the conormal bundle to $Y_{\alpha}$, and consider $\Omega_{\alpha}$ as a sheaf on $T_{Y_{\alpha}}^*(G/B\times G/B)$ via pullback. Then, $[\Omega_{\alpha}]=[\pi_2^*\calO(\alpha)]$ as a sheaf on $T_{Y_{\alpha}}^*$, where $\pi_1$ and $\pi_2$ are the two projections from $T_{Y_{\alpha}}^*(G/B\times G/B)$ to $T^*G/B$ respectively.  
The isomorphism \eqref{eqn:KL} sends the simple generator $\tau_\al$ to $-[\calO_\Delta]-[\Omega_{\alpha}]$, where $\calO_\Delta$ is the structure sheaf of the diagonal component of the Steinberg variety $Z$, and it sends $e^\lambda\in X^*(A)$ to $[\calO_\Delta(\lambda)]$ (see \cite[Prop. 6.1.5]{R08}). This morphism is conjugate to the one in the \textit{loc. cit.} by the sheaf $\calO(\rho)$, and it is related to the one used in \cite{CG97} by an Iwahori--Matsumoto involution (without signs).\footnote{We thank J. Schuermann for pointing this out to us.}

There is a natural embedding of the convolution algebras $K_{G\times \bbC^*}(T^*G/B\times_\calN T^*G/B)$ into $K_{A\times \bbC^*}(T^*G/B\times_\calN T^*G/B)$, which in turn  acts on  $K_T(T^*G/B)$ by convolution from left and from right (see \cite[\S 5.2.20]{CG97}). The left action is given by
\[D_{\alpha}(\calF):=\pi_{1*}(\pi_2^*\calF\otimes \Omega_{\alpha}),\]
where $\calF\in K_T(T^*G/B)$. The pushforward is understood as derived pushforward in equivariant K-theory.  Similarly, the right action is
\[D_{\alpha}'(\calF):=\pi_{2*}(\pi_1^*\calF\otimes \Omega_{\alpha}).\]

For $\calF\in K_T(T^*G/B)$,  the left (resp. right) actions of $\tau_w\in \bbH$ on $\calF$ is denoted  by $T_w(\calF)$ (resp. and $T_w'(\calF)$). 
By definition, $D_\al=-T_\al-1$ and $D_\al'=-T_\al'-1.$

\subsection{Hecke algebra action $D_\al$ on $\stab_{-, T^*G/B, \calL}$}\label{sec:act1}

We will need the following lemma, which can be proved easily by calculating the weights.
\begin{lemma}\label{lem:conormalweight}
For any $v\in W$ and simple root $\alpha$, with $X=T^*G/B$, we have
\begin{gather*}\bigwedge\nolimits^\bullet(T_{(v,v)}T^*_{Y_\alpha})=\bigwedge\nolimits^\bullet(T_vX)\frac{1-e^{v\alpha}}{1-qe^{-v\alpha}},\\
\bigwedge\nolimits^\bullet(T_{(v,vs_\alpha)}T^*_{Y_\alpha})=\bigwedge\nolimits^\bullet(T_{(vs_\alpha, v)}T^*_{Y_\alpha})=\bigwedge\nolimits^\bullet(T_vX)\frac{1-e^{-v\alpha}}{1-qe^{-v\alpha}}.\end{gather*}
\end{lemma}
Among  the alcoves for $\mathfrak{g}$, there is a fundamental one defined by 
\[\nabla:=\{\lambda\in (\Lie A)^*_\bbR|0<(\lambda,\alpha^\vee)<1, \text{for all positive roots } \alpha \}.\]
If we pick the slope $\calL\in \nabla$, we have the following lemma
\begin{lemma}\label{lem:Bruhatcompare}
Given $v>w\in W$ under the Bruhat order, then for any $\xi\in \Lie(A)^*$ in the 
positive chamber, $(\xi, \calL|_v-\calL|_w)<0$.
\end{lemma}
\begin{proof}
By \cite[\S 2]{BGG}, there exits a sequence of positive roots $\alpha_i$, $1\leq i\leq l$, such that $v>vs_{\alpha_1}>\cdots >vs_{\alpha_1}\cdots s_{\alpha_l}=w$. Therefore, $v\alpha_1<0$, $vs_{\alpha_1}\alpha_2<0$, $\dots$, $vs_{\alpha_1}\cdots s_{\alpha_{l-1}}\alpha_l<0$. So
\[(\xi, \calL|_v-\calL|_w)=\sum_i(\xi, \calL|_{vs_{\alpha_1}\cdots s_{\alpha_{i-1}}}-\calL|_{vs_{\alpha_1}\cdots s_{\alpha_{i-1}}s_{\alpha_i}})=\sum_i(\calL, \alpha_i^\vee)(\xi, vs_{\alpha_1}\cdots s_{\alpha_{i-1}}\alpha_i)<0.\]
\end{proof}

In the remaining part of this paper, we fix $\calL\in \Pic(X)\otimes_\bbZ\bbQ$ lying in the fundamental alcove, i.e.,  in the positive chamber and near $0$. Denote 
\begin{equation}\label{eqn:def_stab}
\stab_{-}(w)=\stab_{-,T^*G/B,\calL}(w), \quad \stab_{+}(w)=\stab_{+,T(G/B),\calL^{-1}}(w). 
\end{equation}

\begin{prop}\label{prop:stabim} With notations as above, we have
\[D_{\alpha}(\stab_-(w))=
\left\{ \begin{array}{cc}
-q\stab_-(w)-q^{\frac{1}{2}}\stab_-(ws_{\alpha}),& \text{ if } ws_{\alpha}<w;\\
-\stab_-(w)-q^{\frac{1}{2}}\stab_-(ws_{\alpha}),& \text{ if } ws_{\alpha}>w.
\end{array}\right. \]
\end{prop}

\begin{proof}
By Remark \ref{rem:afterdef}.(2),
\[D_{\alpha}(\stab_-(w))=\sum_v\left(D_{\alpha}(\stab_-(w)),\stab_+(v)\right)\stab_-(v).\]
By the support condition, $\left(D_{\alpha}(\stab_-(w)),\stab_+(v)\right)$ is a proper intersection, so it belongs to $S$, i.e., it is a Laurent polynomial.
By the localization formula and Lemma \ref{lem:conormalweight},
\begin{align*}
&\left(D_{\alpha}(\stab_-(w)),\stab_+(v)\right)\\
=&\sum_{u}\frac{\stab_+(v)|_u\stab_-(w)|_u}{\bigwedge^\bullet(T_{(u,u)}T^*_{Y_{\alpha}})}e^{u\alpha}+\sum_{u}\frac{\stab_+(v)|_{us_\alpha}\stab_-(w)|_u}{\bigwedge^\bullet (T_{(us_{\alpha},u)}T^*_{Y_{\alpha}})}e^{u\alpha}\\
=&\sum_{w\leq u\leq y}\frac{\stab_+(v)|_u\stab_-(w)|_u}{\bigwedge^\bullet(T_uX)}\frac{e^{u\alpha}-q}{1-e^{u\alpha}}\\
&+\sum_{w\leq u, us_{\alpha}\leq y }\frac{\stab_+(v)|_{us_\alpha}\stab_-(w)|_u}{\bigwedge^\bullet(T_uX)}\frac{1-qe^{-u\alpha}}{1-e^{-u\alpha}}e^{u\alpha}.
\end{align*}
We first show that if $v\notin \{w,ws_\alpha\}$, this is 0.

Denote 
\begin{gather*} f_1:=\stab_+(v)|_u\stab_-(w)|_u, \quad f_2:=\bigwedge\nolimits^\bullet(T_uX), \\
~f_3:=\stab_+(v)|_{us_\alpha}\stab_-(w)|_u e^{u\alpha}.\end{gather*}
We can find a common $\xi$ as in Lemma \ref{lem:rigidity} for all $w, u, v$. Then, by the degree condition for stable bases, we have 
\begin{gather*}max_\xi f_1\leq max_\xi(\stab_+(u)|_u)+max_\xi(\stab_-(u)|_u)+(\xi, \calL|_u-\calL|_w+\calL^{-1}|_u-\calL^{-1}|_v)\\
=(\xi,2\rho+\calL|_v-\calL|_w),\\
max_\xi f_2=(\xi, 2\rho),\quad max_\xi f_3\leq  (\xi, 2\rho+\calL|_u-\calL|_w+\calL|_v-\calL|_{us_\alpha}),\end{gather*}
where the last inequality follows from the degree condition for stable bases and Equation (\ref{eq:degalpha}).

By Lemma \ref{lem:Bruhatcompare}, $(\xi,\calL|_v-\calL|_w)<0$, and $(\xi, \calL|_u-\calL|_w+\calL|_v-\calL|_{us_\alpha})<0$ because $u>w$ and $v>us_\alpha$. Therefore,
\begin{equation} \label{eq:lim1}
\lim_{t\rightarrow \infty}\left(D_{\alpha}(\stab_-(w)),\stab_+(v)\right)(t\xi)=0.
\end{equation}

For the minimal degree, we have
\begin{gather*}min_\xi f_1\geq min_\xi(\stab_+(u)|_u)+min_\xi(\stab_-(u)|_u)\\
+(\xi, \calL|_u-\calL|_w+\calL^{-1}|_u-\calL^{-1}|_v)=(\xi,-2\rho+\calL|_v-\calL|_w),\\
min_\xi f_2=(\xi, -2\rho),\quad min_\xi f_3\geq  (\xi, -2\rho+\calL|_u-\calL|_w+\calL|_v-\calL|_{us_\alpha}),\end{gather*}
where the last inequality follows from the degree condition for stable bases and Equation (\ref{eq:degalpha2}). 

We can choose $\calL$ sufficiently close to 0, such that 
\[-1<(\xi, \calL|_v-\calL|_w)<0, \quad \text{ and }-1<(\xi, \calL|_u-\calL|_w+\calL|_v-\calL|_{us_\alpha})<0.\]
Then,
\begin{equation}\label{eq:lim2}
\lim_{t\rightarrow -\infty }\left(D_{\alpha}(\stab_(w)),\stab_+(v)\right)(t\xi) \text{ is bounded}.
\end{equation}
Due to Lemma \ref{lem:rigidity} and  \eqref{eq:lim1}, \eqref{eq:lim2}, we get 
\[\left(\stab_-(w)),\stab_+(v)\right)=0, \text{ if } v\notin \{w,ws_\alpha\}.\]
Hence we only need to compute 
\[\left(D_{\alpha}(\stab_-(w)),\stab_+(w)\right) \text{ and } \left(D_{\alpha}(\stab_-(w)),\stab_+(ws_\alpha)\right).\]
This is done by analyzing two cases below, depending on the order of $w$ and $ws_{\alpha}$.

\noindent (1). {\it Case $ws_{\alpha}<w$:}

There is only one term in the localization of $\left(D_{\alpha}(\stab_-(w)),\stab_+(ws_\alpha)\right)$. Therefore, by Lemma \ref{lem:charstab} and Lemma \ref{lem:conormalweight}, we get
\begin{align*}
&\left(D_{\alpha}(\stab_-(w)),\stab_+(ws_\alpha)\right)\\
=&\frac{\stab_+(ws_\alpha)|_{ws_\alpha}\stab_-(w))|_w}{\bigwedge^\bullet (T_{(ws_{\alpha},w)}T^*_{Y_{\alpha}})}e^{w\alpha}=-q^{\frac{1}{2}}.
\end{align*}
There are two terms in the localization of $\left(D_{\alpha}(\stab_-(w)),\stab_+(w)\right)$.
\begin{align*}
&\left(D_{\alpha}(\stab_-(w)),\stab_+(w)\right)\\
=&\frac{\stab_+(w)|_w\stab_-(w))|_w}{\bigwedge^\bullet (T_{(w,w)}T^*_{Y_{\alpha}})}e^{w\alpha}+\frac{\stab_+(w)|_{ws_\alpha}\stab_-(w)|_w}{\bigwedge^\bullet (T_{(ws_{\alpha},w)}T^*_{Y_{\alpha}})}e^{w\alpha}\\
=&\frac{e^{w\alpha}-q}{1-e^{w\alpha}}+\frac{\stab_+(w)|_{ws_\alpha}\stab_-(w)|_w}{\bigwedge^\bullet (T_wX)}\frac{1-qe^{-w\alpha}}{1-e^{-w\alpha}}e^{w\alpha}.
\end{align*}
As in the first part of the proof, we can find a $\xi\in \Lie A$ in the positive chamber, such that $-1<(\xi, \calL|_w-\calL|_{ws_\alpha})<0$. Notice that $w\alpha<0$. We have
\begin{gather*}\lim_{t\rightarrow \infty}\left(D_{\alpha}(\stab_-(w)),\stab_+(w)\right)(t\xi)= -q,\\
\lim_{t\rightarrow -\infty}\left(D_{\alpha}(\stab_-(w)),\stab_+(w)\right)(t\xi)\text{ is bounded }.\end{gather*}
Therefore, due to Lemma \ref{lem:rigidity}, we get
\[\left(D_{\alpha}(\stab_-(w)),\stab_+(w)\right)=-q.\]

\noindent(2). {\it Case $ws_{\alpha}>w$:}

Although this case can be proved directly using the relation $D_{\alpha}^2+(q+1)D_{\alpha}=0$, we still give a localization proof for it.
\begin{align*}
&\left(D_{\alpha}(\stab_-(w)),\stab_+(w)\right)\\
=&\frac{e^{w\alpha}-q}{1-e^{w\alpha}}+\frac{\stab_+(w)|_{w}\stab_-(w)|_{ws_\alpha}}{\bigwedge^\bullet (T_wX)}\frac{1-qe^{-w\alpha}}{1-e^{-w\alpha}}e^{-w\alpha}.
\end{align*}
As in the first case, the limit as $t\rightarrow +\infty$ is $-1$, while the limit as $t\rightarrow -\infty$ is bounded. Therefore,
\[\left(D_{\alpha}(\stab_-(w)),\stab_+(w)\right)=-1.\]
For the other one, we have
\begin{align*}
&\left(D_{\alpha}(\stab_-(w)),\stab_+(ws_\alpha)\right)\\
=&\frac{\stab_+(ws_\alpha)|_w\stab_-(w)|_{w}}{\bigwedge^\bullet (T_wX)}\frac{e^{w\alpha}-q}{1-e^{w\alpha}}\\
&+\frac{\stab_+(ws_\alpha)|_{ws_\alpha}\stab_-(w)|_{ws_\alpha}}{\bigwedge^\bullet (T_{ws_\alpha}X)}\frac{e^{-w\alpha}-q}{1-e^{-w\alpha}}\\
&+\frac{\stab_+(ws_\alpha)|_{ws_\alpha}\stab_-(w)|_{w}}{\bigwedge^\bullet (T_wX)}\frac{1-qe^{-w\alpha}}{1-e^{-w\alpha}}e^{w\alpha}\\
&+\frac{\stab_+(ws_\alpha)|_{w}\stab_-(w)|_{ws_\alpha}}{\bigwedge^\bullet (T_wX)}\frac{1-qe^{-w\alpha}}{1-e^{-w\alpha}}e^{-w\alpha}.
\end{align*}
Because of Lemma \ref{lem:charstab}, the third term is $-q^{-\frac{1}{2}}\frac{1-qe^{w\alpha}}{1-e^{w\alpha}}\frac{1-qe^{-w\alpha}}{1-e^{-w\alpha}}$. Since $ws_\alpha>w$, $w\alpha>0$. As in the first case, pick good $\xi$, then the limit as $t\rightarrow +\infty$ is $-q^{\frac{1}{2}}$, while the limit as $t\rightarrow -\infty$ is bounded.
Therefore,
\[\left(D_{\alpha}(\stab_-(w)),\stab_+(ws_\alpha)\right)=-q^{\frac{1}{2}}.\]
\end{proof}

\subsection{Hecke algebra action $D_\al'$ on $\stab_{+, T(G/B), \calL^{-1}}$}\label{sec:act2}
In this section, we compute the second Hecke algebra $\bbH$ action on the stable bases for the positive chamber $+$. Although the method from~\ref{sec:act1} still works in this case,  we use a different method for illustration purpose. 

The relation between these two Hecke actions is the following adjoint property, which is a K-theoretic analogue of \cite[Lemma 5.2]{AMSS17}.
\begin{lemma}\label{lem:geoadjoint}
For any $\calF$ and $\calG$ in $K_T^*(T^*G/B)$, we have
\[(D_\alpha(\calF), \calG)=(\calF, D_\alpha'(\calG)).\]
Therefore, for any $T_w\in \bbH$, 
\[(T_w(\calF), \calG)=(\calF, T_{w^{-1}}'(\calG)).\]
\end{lemma}
By definition, under this pairing, operators from the subalgebra $K_{G\times\bbC^*}(T^*G/B)\subseteq K_{G\times\bbC^*}(T^*G/B\times_\calN T^*G/B)$ are self-adjoint.
\begin{proof}[Proof of Lemma~\ref{lem:geoadjoint}]
We only need to prove the first one. And we can check on the fixed point basis. Using localization and Lemma \ref{lem:conormalweight}, we get
\begin{equation}\label{eq:Heckeaction}D_\alpha(\iota_{v*}1)=\frac{e^{v\alpha}-q}{1-e^{v\alpha}}\iota_{v*}1+\frac{e^{v\alpha}-q}{1-e^{-v\alpha}}\iota_{vs_\alpha*}1,\quad D_\alpha'(\iota_{v*}1)=\frac{e^{v\alpha}-q}{1-e^{v\alpha}}\iota_{v*}1+\frac{1-qe^{-v\alpha}}{e^{v\alpha}-1}\iota_{vs_\alpha*}1.\end{equation}
Therefore,
\begin{gather*}(D_\alpha(\iota_{v*}1),\iota_{w*}1)=\delta_{v,w}\frac{e^{v\alpha}-q}{1-e^{v\alpha}}\bigwedge\nolimits^\bullet(T_vT^*G/B)+\delta_{vs_\alpha,w}\frac{e^{v\alpha}-q}{1-e^{-v\alpha}}\bigwedge\nolimits^\bullet(T_{vs_\alpha}T^*G/B),\\
(\iota_{v*}1,D_\alpha'(\iota_{w*}1))=\delta_{v,w}\frac{e^{v\alpha}-q}{1-e^{v\alpha}}\bigwedge\nolimits^\bullet(T_vT^*G/B)+\delta_{vs_\alpha,w}\frac{1-qe^{-w\alpha}}{e^{w\alpha}-1}\bigwedge\nolimits^\bullet(T_vT^*G/B).\end{gather*}
Now it is easy to see they are equal to each other.
\end{proof}

\begin{theorem}\label{thm:stabpim}
With notations defined in \eqref{eqn:def_stab}, the affine Hecke algebra $\bbH$ acts on the stable bases as follows:
\begin{gather*}
T_{s_\alpha}(\stab_-(w))=
\left\{ \begin{array}{cc}
q^{\frac{1}{2}}\stab_-(ws_{\alpha})+(q-1)\stab_-(w),& \text{ if } ws_{\alpha}<w;\\
q^{\frac{1}{2}}\stab_-(ws_{\alpha}),& \text{ if } ws_{\alpha}>w.
\end{array}\right. \\
T_{s_\alpha}'(\stab_+(w))=
\left\{ \begin{array}{cc}
q^{\frac{1}{2}}\stab_+(ws_{\alpha})+(q-1)\stab_+(w),& \text{ if } ws_{\alpha}<w;\\
q^{\frac{1}{2}}\stab_+(ws_{\alpha}),& \text{ if } ws_{\alpha}>w.
\end{array}\right. 
\end{gather*}
\end{theorem}
\begin{proof}
The formula concerning the $T_{s_{\al}}$ action comes from the identity  $T_{s_\alpha}=-D_\alpha-1$  and Proposition \ref{prop:stabim}.

We look at the $T'_{s_\al}$ action.
By the duality of stable bases (see Remark \ref{rem:afterdef}) and Proposition \ref{prop:stabim}, we have
\begin{align*}
T_{s_\alpha}'(\stab_+(w))=&\sum_y\left(T_{s_\alpha}'(\stab_+(w)), \stab_-(y)\right)\stab_+(y)\\
=&\sum_y\left(\stab_+(w), T_{s_\alpha}(\stab_-(y))\right)\stab_+(y)\\
=&\left(\stab_+(w), T_{s_\alpha}(\stab_-(w))\right)\stab_+(w)\\
+&\left(\stab_+(w), T_{s_\alpha}(\stab_-(ws_\alpha))\right)\stab_+(ws_\alpha).
\end{align*}
The rest follows from the first part of this theorem and the duality property in Remark \ref{rem:afterdef}.(2).
\end{proof}

\subsection{A recursive formula of the restriction coefficients}
Using the Hecke actions, we  give a recursive formula for the restriction coefficients of $\stab_-(w)$. 
In Theorem \ref{thm:1}, we  will give a closed formula of those coefficients.

\begin{prop}\label{lem:strcharm}
With notations defined in \eqref{eqn:def_stab}, the restriction coefficients $\stab_-(w)|_v$ are uniquely characterized by
\begin{enumerate}
\item 
$\stab_-(w)|_v=0$, unless $v\geq w$.
\item 
$\stab_-(w)|_w=q_{w}^{\frac{1}{2}} \underset{\al\in \Sigma^-\cap w\Sigma^-}\prod(1-qe^{\al})\cdot \underset{{\al\in \Sigma^+\cap w\Sigma^-}}\prod(1-e^{\al})$.
\item 
\[\stab_-(w)|_{vs_\alpha}=
\left\{ \begin{array}{cc}
\frac{(1-q)e^{v\alpha}}{1-qe^{-v\alpha}}\stab_-(w)|_v+q^{\frac{1}{2}}\frac{1-e^{v\alpha}}{1-qe^{-v\alpha}}\stab_-(ws_\alpha)|_v, & \text{ if } ws_{\alpha}<w;\\
\frac{1-q}{1-qe^{-v\alpha}}\stab_-(w)|_v+q^{\frac{1}{2}}\frac{1-e^{v\alpha}}{1-qe^{-v\alpha}}\stab_-(ws_\alpha)|_v,& \text{ if } ws_{\alpha}>w.
\end{array}\right. \]
\end{enumerate}
\end{prop}
This is an analogue of Corollary 3.3 in \cite{Su15}.
\begin{proof}
The uniqueness can be easily proved by induction on $\ell(y)$. The first two equalities now follow directly from Lemma \ref{lem:charstab}. 

The last equality follows from Proposition \ref{prop:stabim} and \eqref{eq:Heckeaction} by applying $D_\alpha$ to the following identity
\[\stab_-(w)=\sum_v \stab_-(w)|_v\frac{\iota_{v*}1}{\bigwedge\nolimits^\bullet(T_vT^*G/B)}.\] 
\end{proof}
A similar recursive formula for $\stab_+(w)|_y$ can also be obtained from Theorem \ref{thm:stabpim}.

\section{More on the affine Hecke algebra}\label{sec:affHecke}

In this section we recall the definition of the affine Hecke algebra in terms of the twisted group algebra of Kostant and Kumar \cite{KK90}, while following notions from \cite{CZZ1, CZZ2}. The root system we consider will be the one associated to the group $G$. 

\subsection{The Demazure-Lusztig elements}
In the ring $S=R[\La]$, we use the following notations. 
\[
x_\al=1-e^{-\al},  ~ x_{-\al}=-e^\al x_\al, ~\tilx_\al=q-e^\al, ~\hatx_{\al}=-e^{-\al}\tilx_\al=1-qe^{-\al},  ~q_w=q^{\ell(w)}.\]
\begin{eqnarray}\label{eq:1} ~x_{\pm w}=\prod_{\be\in \Sigma_w}x_{\pm \beta}, ~\tilx_{\pm w}=\prod_{\be\in \Sigma_w}\tilx_{\pm\beta}, ~\hatx_{\pm w}=\prod_{\be\in \Sigma_w}\hatx_{\pm\beta}.
\end{eqnarray}
Note that $u(x_\la)=x_{u(\la)}$ for $u\in W, \la\in \La$, but $u(x_w)\neq x_{uw}$.

Consider the twisted product $Q_W=Q\rtimes R[W]$, 
which has a $Q$-basis $\{\de_w\}_{w\in W}$. The ring $Q_W$ naturally acts on $Q$ by 
\[
p\de_w\cdot p'=pw(p'), ~p,p'\in Q.
\]
For each root $\al$, we define the push-pull element 
\[
Y_\al=\frac{1}{x_{-\al}}+\frac{1}{x_\al}\de_\al=\frac{1}{1-e^\al}+\frac{1}{1-e^{-\al}}\de_\al.
\]
We also define the divided difference operator (or the Demazure operator) 
\[
\Dem_\al(p):=\frac{s_\al(p)-p}{1-e^{-\al}}=(Y_\al-1)\cdot p, ~p\in Q.\]
It restricts to an $S^W$-linear endomorphism  on $S$.  We define the  Demazure-Lusztig elements:
\begin{eqnarray}\label{definition of tau}
\taup_\al=\tilx_\al Y_\al-1=\frac{q-1}{x_{-\al}}+\frac{\tilx_\al}{x_\al}\de_\al=\frac{q-1}{1-e^\al}+\frac{q-e^\al}{1-e^{-\al}}\de_\al,\\
\taum_\al=\hatx_\al Y_{-\al}-1=\frac{q-1}{x_{-\al}}+\frac{\hatx_\al}{x_{-\al}}\de_\al=\frac{q-1}{1-e^\al}+\frac{1-qe^{-\al}}{1-e^\al}\de_\al.
\end{eqnarray}
For simplicity, for simple root $\al_i$, we will write $x_{\pm i}, \tilx_{\pm i}, \hatx_{\pm i}, Y_{\pm i}, \Dem_{\pm i}, \taupm_{ i}$ for $x_{\pm \al_i}, \tilx_{\pm \al_i}, \hatx_{\pm \al_i}, Y_{\pm\al_i}, \Dem_{\pm\al_i}$ and $\taupm_{\al_i}$, respectively. For each reduced  decomposition  $w=s_{i_1}\cdots s_{i_k}$,  define $\taupm_{w}=\taupm_{i_1}\cdots \taupm_{i_k}$, and similarly define $Y_{\pm w}$ (e.g., $Y_{-w}$ is a product of $Y_{-i}$). As shown in \cite{KK90}, they do not depend on the choice of the reduced decomposition.
 
By straightforward computations, we have the following properties:
\begin{lemma} \label{lem:pro}
\begin{enumerate}
\item $Y_i^2=Y_i$,  $Y_i p-s_i(p)Y_i=-\Dem_{-i}(p), ~p\in Q$.
\item $\taupm_i p-s_i(p)\taupm_i=-(q-1) \Dem_{-i}(p), ~p\in Q.$
\item $\taupm_i^2=(q-1)\taupm_i+q$, $\taupm_i^{-1}=q^{-1}(\taupm_i+1-q)$.
\item $\de_i=x_iY_i-\frac{x_i}{x_{-i}}=\frac{x_{-i}}{\hatx_{i}}\taum_i-\frac{q-1}{\hatx_{i}}=\frac{x_{i}}{\tilx_i}\taup_i-\frac{(q-1)x_{i}}{x_{-i}\tilx_i}. $
\end{enumerate}
\end{lemma}

\begin{lemma}\label{lem:basis}\begin{enumerate}
\item We have $\taupm_{w}=\sum_{v\le w}\apm_{w,v}\de_v$ and  $ \de_w=\sum_{v\le w}\bpm_{w,v}\taupm_{v}$ such that
\[ \apm_{w,v}\in S[\frac{1}{x_{w_0}}], \quad \bpm_{w,v}\in S[\frac{1}{\hatx_{w_0}}], \quad \bp_{w,w}=\frac{1}{\ap_{w,w}}=\frac{x_{w}}{\tilx_w}, \quad \bm_{w,w}=\frac{1}{\am_{w,w}}=\frac{x_{-w}}{\hatx_w}.\]
\item We have $\taupm_w=\sum_{v\le w}d^\pm_{w,v}Y_{\pm v}$ such that 
\[\quad d^\pm_{w,v}\in S, \quad d^+_{w,w}=\tilx_w,\quad d^-_{w,w}=\hatx_w,\quad d^\pm_{w,e}=(-1)^{\ell(w)}.\]
\end{enumerate}
\end{lemma}
\begin{proof} Similar to \cite[Lemma 3.2]{CZZ2}, these identities follow from Lemma \ref{lem:pro}.
\end{proof}

\begin{remark}
(1).   There is an anti-involution on $Q_W$ defined by
\[
\iota:Q_W\to Q_W, \quad p\de_v\mapsto \de_{v^{-1}}p\frac{v(x_{-w_0} \hatx_{w_0})}{x_{-w_0} \hatx_{w_0}}, \quad p\in Q, 
\]
such that 
\[
\iota(\taupm_\al)=\taump_\al. 
\]

(2).  Recall the following operator  of Lusztig in \cite{Lus85}
\[T^L_\alpha=\frac{q-1}{1-e^{\alpha}}+\frac{1-qe^{\alpha}}{1-e^{\alpha}}\de_\alpha.\]
They satisfy the same relations as $\tau_\al$ do. We have identities
\begin{equation}\label{equ:relationwithlus}
e^{-\rho} \taup_\alpha e^{\rho}=e^\rho \taum_\alpha e^{-\rho}=-q\cdot (T_\al^L|_{q\to q^{-1}}).
\end{equation}

\end{remark}

\subsection{The Hecke algebra}
\begin{definition}Define the affine $0$-Hecke algebra $\DF$ to be the $R$-subalgebra generated by $S$ and $Y_i$ for all $i$. Define the  affine Hecke algebra $\bbH$ to be the $R$-subalgebra of $Q_W$ generated by $S$ and  $\taup_i$ for all $i$. It is not difficult to see that  all $\taum_i$ together with $S$ also generate $\bbH$.
\end{definition}

\begin{lemma}\cite{KK90, Lus85}
The sets $\{Y_{w}\}_{w\in W}$ , $\{\taup_{w}\}_{w\in W}$ and $\{\taum_{w}\}_{w\in W}$  are $Q$-bases of $Q_W$. Moreover, the first set is a $S$-basis of $\bbD$ and the last two are $S$-bases of $\bbH$.
\end{lemma}
\begin{proof} The first statement follows from Lemma \ref{lem:basis}, and the second one is from \cite[Proposition 7.7]{CZZ1} and \cite[Corollary 3.4]{ZZ14}. 
\end{proof}
The following lemma is used to give an algebraic proof of duality of stable bases in Theorem \ref{thm:main}.
\begin{lemma}\label{lem:dualcoeff}
 Writing $\taupm_w(\taupm_{w_0u})^{-1}=\sum_{v\in W}c_v\taupm_v$,  then $c_{w_0}=q_{w_0u}^{-1}\de_{w,u}$.
\end{lemma}
\begin{proof}This is a special case of \cite[Proposition 3]{NN15}. More precisely,  in loc.it., letting $t_1=q^{\frac{1}{2}}, t_2=-q^{-\frac{1}{2}}$, then one can identify $h_i$ with $q^{-\frac{1}{2}}\taupm_i$ and $\hat{h}_i$ with $q^{\frac{1}{2}}\taupm_i^{-1}$, and the conclusion  follows.
\end{proof}

\section{Algebraic description of Stable bases}\label{sec:algstable}

In this section, we first briefly recall the algebraic models of $K_T(G/B)$, $K_T(T^*G/B)$ and the morphisms between them (details can be found in \cite{KK90, CZZ2, CZZ3}).  We then obtain a formula of stable bases in this algebraic setting. 

\subsection{The dual of the twisted group algebra}
Define
\[\DFd:=\Hom_S(\DF, S)\subset Q_W^*=\Hom(W,Q).\] 
Let $\{f_w\}_{w\in W}$ be the standard basis of  $Q_W^*$,  that is, $f_w(\de_v)=f_w(v)=\de_{w,v}$. There is a commutative product with identity:
\[ f_w\cdot f_v=\de_{w,v}f_v, \quad \unit:=\sum_{w\in W}f_w\in \DFd\subset Q_W^*.\]
Indeed,  $Q_W^*$ is a commutative $Q$-algebra and $\DFd$ is a commutative $S$-algebra.

There is a canonical action of $Q_W$ on $Q_W^*$ defined as follows:
\[
(z\bullet f)(z')=f(z'z), \quad z,z'\in Q_W, f\in Q_W^*.
\]
We will frequently use the following identities, whose proof can be checked easily, or found in \cite[\S6]{CZZ2}.
\begin{equation}\label{eq:bullet}
p\bullet f_v=v(p)f_v, \quad \de_w\bullet f_v=f_{vw^{-1}}, \quad p\bullet (f\cdot g)=(p\bullet f)\cdot g=f\cdot (p\bullet g), \quad p\in S, ~f,g\in Q_W^*.
\end{equation}
The action indeed restricts to an action of $\DF$ on $\DFd$. Moreover, it induces an action of $W\subset Q_W$ on $Q_W^*$. The $W_J$-invariant $Q$-submodule $(Q_W^*)^{W_J}$ has a basis $\{\sum_{v\in W_J}f_{wv}\}_{w\in W^J}$.

For each $J\subset \Pi$, define the following elements in $Q_W$:
\begin{eqnarray}
Y_{\Pi/J}&=&\sum_{w\in W^J}\de_w\frac{x_{-w_0^J}}{x_{-w_0}}=\sum_{w\in W^J}\de_w\prod_{\al\in \Sigma^+\backslash \Sigma^+_J}\frac{1}{1-e^{\al}},\\
Y_J&=&\sum_{w\in W_J}\de_w \frac{1}{x_{-w_0^J}}=\sum_{w\in W_J}\de_w\prod_{\al\in \Sigma_J^+}\frac{1}{1-e^\al}, \\
 \hatY_{\Pi/J}&=&\sum_{w\in W^J}\de_w \frac{x_{-w_0^J}\hatx_{w_0^J}}{x_{-w_0}\hatx_{w_0}}=\sum_{w\in W^J}\de_w\prod_{\al\in \Sigma^+\backslash\Sigma_J^+}\frac{1}{(1-e^{\al})(1-qe^{-\al})},\\
 \label{eq:YJ}  \hatY_J&=&\sum_{w\in W_J}\de_w \frac{1}{x_{-w_0^J}\hatx_{w_0^J}}=\sum_{w\in W_J}\de_w\prod_{\al\in \Sigma_J^+}\frac{1}{(1-e^{\al})(1-qe^{-\al})}.
\end{eqnarray}
In particular, $Y_\Pi$ and $\hatY_\Pi$ are defined when $J=\Pi$. Similar as \cite[Lemmas 5.7 and 6.4]{CZZ2}, we have  the composition rule and the Projection Formula:
\begin{gather}
\label{eq:comp} Y_{\Pi/J}Y_J=Y_\Pi, \quad \hatY_{\Pi/J}\hatY_J=\hatY_\Pi,\\
\label{eq:projfor}\quad f\cdot( Y_J\bullet f')=Y_J(ff'), ~ f\cdot (\hatY_J\bullet f')=\hatY_J\bullet (ff'), ~ f'\in Q_W^*, ~f\in (Q_W^*)^{W_J} .
\end{gather}

Via the embedding $\bbH\subset Q_W$, we can restrict the $\bullet$-action to a left action of $\bbH$ on $Q_W^*$. On the other hand, $\bbH$ also acts on the right on $Q_W^*$, where the action of $\taup_w\in \bbH$ on $f\in Q_W^*$ is given  by 
\[
\taup_{w^{-1}}\bullet f. 
\]
The $\bullet$-action is a well-defined action of $\bbH$ on $Q_W^*$, which is linear with respect to the $Q$-module structure on $Q_W^*$ coming from $\Hom(W,Q)$, hence so is the right action defined above. Indeed, $\taum_\al\bullet\_$ and $\taup_\al\bullet\_$ correspond to the $T_\al$ and $T_\al'$ actions in Section \ref{sec:Heckeaction}, respectively (see Lemma \ref{lem:Ttau}). The following lemma is the algebraic model of Lemma \ref{lem:geoadjoint}:
\begin{lemma}[Adjointness]\label{lem:adjoint}For any $\al_i\in J, f,g\in Q_W^*$, we have 
\begin{gather*}
\hatY_J\bullet ((\taup_i\bullet f)\cdot  g)=\hatY_J\bullet (f\cdot (\taum_i\bullet g)).
\end{gather*}
\end{lemma}
\begin{proof}Note that the $\bullet$-action is $Q$-linear, so it suffices to assume that $f=f_v, g=f_u$ with $u,v\in W$. The identity then follows from direct computation. 
\end{proof}
We get an easy corollary of the coefficients appearing in Lemma \ref{lem:basis}.

\begin{lemma} \label{lem:invert}We have 
\[
\apm_{w^{-1},v}v(x_{-w_0})v(\hatx_{w_0})=v(\amp_{w,v^{-1}})x_{-w_0}\hatx_{w_0}.
\]
\end{lemma}
\begin{proof}From Lemma \ref{lem:adjoint} we know that 
\[
\hatY_\Pi\bullet ((\taupm_w\bullet f_e)\cdot f_v)=\hatY_\Pi\bullet (f_e\cdot (\taump_{w^{-1}}\bullet f_v)).
\]
Direct computations using Lemma \ref{lem:basis} and \eqref{eq:bullet}  shows that the left hand side is $\frac{v(\apm_{w,v^{-1}})}{v(x_{-w_0})v(\hatx_{w_0})}\unit$, and the right hand side is equal to $\frac{\amp_{w^{-1},v}}{x_{-w_0}\hatx_{w_0}}\unit$. Hence, the conclusion follows.
\end{proof}

\subsection{An algebraic model of  stable bases}\label{relation}
We define ($e$ being the identity element of $W$) 
\begin{equation}\label{definition of extreme points}
 \pt:=\pt_e=x_{-w_0} f_e, \quad  \pt_{w_0}:=x_{-w_0}f_{w_0}.
\end{equation}
Both of them belong to $\DFd$ (\cite[Lemma 10.3]{CZZ2}). They can be viewed (up to certain normalization) as the push-forward of the fundamental class  in $K_A(G/B)$ along the $A$-fixed points $e, w_0\in W$.

\begin{definition}\label{def:stabalg}
Define $\Stp_{w}=\taup_{w^{-1}}\bullet \pt_e$ and $\Stm_{u}=(\taum_{w_0u})^{-1}\bullet \pt_{w_0}$. 
\end{definition}
By definition it is easy to see that if $\ell(ws_i)\ge \ell(w)$, then,
\[\taupm_i\bullet\Stpm_w=\Stpm_{ws_i}. \]
Therefore, for any $w\in W$, $\Stpm_w=\taupm_{w^{-1}}\bullet \Stpm_e.$ Note that $\Stp_e=\pt_e$, $\Stm_{w_0}=\pt_{w_0}$.

By the standard theory of Kostant-Kumar, $K_T(G/B)\cong \DFd$ and $ K_T(T^*G/B)\otimes_{S}Q\cong Hom(W,Q)=Q_W^*$ \cite[Theorem (3.13)]{KK90}. Let $p:T^*G/B\to G/B$  be the canonical projections. Then the isomorphism 
$Q_W^*\cong K_T(G/B)\otimes_SQ\overset{p^*}{\underset\sim\longrightarrow}K_T(T^*G/B)\otimes_SQ$ is given by  the formula
$p^*=\hatx_{w_0}\bullet\_$; the map
$K_T(T^*G/B)\underset{\sim}{\overset{(p^*)^{-1}}\longrightarrow} K_T(G/B)\longrightarrow K_T(\bbC).$
is given by the formula $\hatY_\Pi\bullet\_$.

For any $\calF\in K_T(T^*G/B)\subset Q_W^*$, we can write $\calF=\sum_w\calF|_wf_w\in Q_W^*$ with $\calF|_w\in S. $ For example,   $\iota_{w*}1=w(x_{-w_0}\hatx_{w_0})f_w$, where $\iota_w:\Spec(\bbC)\to T^*G/B$ is the embedding of the $T$-fixed point corresponding to $w\in W$, and 
\[w(x_{-w_0}\hatx_{w_0})=\prod_{\alpha>0}(1-e^{w\alpha})(1-qe^{-w\alpha})=\bigwedge\nolimits^\bullet T_w(T^*G/B).\]

\begin{theorem}\label{cor:alggeo}
Under the above identifications, for any $u\in W$ we have
\[ \stab_{+}(u)=q_u^{-\frac{1}{2}}\St^+_u, \quad \stab_{-}(u)=q_{w_0}q_u^{-\frac12}\St^-_u.\]
\end{theorem}

Now we prove this theorem. First of all, we have the following relation between $\taupm_\alpha$ in  (\ref{definition of tau}) and the above operators $T_{s_\alpha}$, $T_{s_\alpha}'$.
\begin{lemma}\label{lem:Ttau}
As operators on $K_T(T^*G/B)$, we have
\[T_{s_\alpha}'=\taup_\al, \quad T_{s_\alpha}= \taum_\al .\]
\end{lemma}
\begin{proof}
By \eqref{eq:Heckeaction}, the operators $T_{s_\alpha}=-D_\alpha-1$ and $T_{s_\alpha}'=-D_\alpha'-1$ act on the basis $\{f_v=\frac{\iota_{v*}1}{\bigwedge^\bullet(T_vT^*G/B)}|v\in W\}$ as follows:
\[T_\alpha(f_v)=\frac{q-1}{1-e^{v\alpha}}f_v+\frac{1-qe^{v\alpha}}{1-e^{-v\alpha}}f_{vs_\alpha},\quad T_\alpha'(f_v)=\frac{q-1}{1-e^{v\alpha}}f_v+\frac{q-e^{-v\alpha}}{1-e^{v\alpha}}f_{vs_\alpha}.\]
Comparing with the $\bullet$-action of $\taupm_\al$ on $f_v\in Q_W^*$ using \eqref{eq:bullet}, we get the conclusion. 
\end{proof}

\begin{proof}[Proof of Theorem~\ref{cor:alggeo}]
Let us consider the first identity. By Lemma~\ref{lem:charstab}, Equation~(\ref{definition of extreme points}), and  Definition~\ref{def:stabalg}, we have
\[\stab_{-}(w_0)=q_{w_0}^{\frac{1}{2}}\prod_{\beta>0}(1-e^\beta)f_{w_0}=q_{w_0}^{\frac12}\pt_{w_0}=q_{w_0}^{\frac12}\Stm_{w_0},\]
and moreover, $\St_u^-=(\taum_{w_0u})^{-1}\bullet \pt_{w_0}$. On the other hand, by Proposition \ref{prop:stabim}, we get
\[T_{w_0u}(\stab_{-}(u))=q_{w_0u}^{\frac{1}{2}}\stab_{-}(w_0).\]
Using Lemma \ref{lem:Ttau}, we get
\begin{align*}
\St^-_u=(\taum_{w_0u})^{-1}\bullet \pt_{w_0}=(T_{w_0u})^{-1}(q_{w_0}^{-\frac12}\stab_{-}(w_0))
=q_{w_0}^{-\frac12}q_{w_0u}^{-\frac12}\stab_{-}(u)=q_{w_0}^{-1}q_u^{\frac{1}{2}}\stab_{-}(u).
\end{align*}
This proves the formula for $\stab_{-}(u)$.

Lemma \ref{lem:charstab},    (\ref{definition of extreme points}) and Definition \ref{def:stabalg} show that 
\[\stab_{+}(e)=\prod\limits_{\beta>0}(1-e^\beta)f_e=\St^+_e.\]
Moreover, Theorem \ref{thm:stabpim} shows that  $T_{u^{-1}}'(\stab_{+}(e))=q_u^{\frac{1}{2}}\stab_{+}(u)$. Comparing with Definition \ref{def:stabalg} and using Lemma \ref{lem:Ttau}, we get the formula for $\stab_{+}(u)$.
\end{proof}

\subsection{The duality}
Let $\taupm_w^*$ be the bases of $Q_W^*$ dual to $\taupm_{w}$, then $
\taupm_{w}^*=\sum_{v\ge w}b^\pm_{v,w}f_v$ by Lemma \ref{lem:basis}.
\begin{lemma}\label{lem:dualT}The map $ Q_W^*\times Q_W^*\to Q, (f,g)\mapsto \hatY_\Pi\bullet (fg)$ defines a perfect pairing and the basis $\hatx_{w_0}\taump_{v}^*$ is dual to the basis $\taupm_{w^{-1}}\bullet \pt_e$. In particular, $\taum_v^*$ is dual to $\Stp_w.$
\end{lemma}
\begin{proof}By  Lemma \ref{lem:basis}  we know  $\taump_{v}^*=\sum_{v'}\bmp_{v',v}f_{v'}$ and  $\taupm_{w^{-1}}=\sum_{u}\apm_{w^{-1}, u}\de_u$, so 
\begin{flalign*}
\taupm_{w^{-1}}\bullet \pt_e&=(\sum_u \apm_{w^{-1}, u}\de_u)\bullet (x_{-w_0}f_e)
=\sum_ux_{-w_0}u^{-1}(\apm_{w^{-1}, u})f_{u^{-1}}\\ & =\sum_ux_{-w_0}u(\apm_{w^{-1}, u^{-1}})f_{u}\overset{\sharp_1}=\sum_u\frac{\amp_{w,u}u(x_{-w_0})u(\hatx_{w_0})}{\hatx_{-w_0}}f_u.
\end{flalign*}
Here $\sharp_1$ follows from Lemma \ref{lem:invert}. 
According to \eqref{eq:bullet}, we have
\begin{flalign*}
\hatY_\Pi\bullet [(\hatx_{w_0}\taump_v^*)\cdot (\taupm_{w^{-1}}\bullet \pt)]&=\hatY_\Pi\bullet(\sum_{v'}\hatx_{w_0}\bmp_{v',v}f_{v'}\cdot \sum_u\frac{\amp_{w,u}u(x_{-w_0})u(\hatx_{w_0})}{\hatx_{w_0}}f_u)\\
&=\hatY_\Pi\bullet (\sum_u\bmp_{u,v} \amp_{w,u}u(x_{-w_0})u(\hatx_{w_0})f_u)\\
&=\sum_{w'\in W}\de_{w'}\frac{1}{x_{-w_0}\hatx_{w_0}}\bullet  (\sum_u\bmp_{u,v} \amp_{w,u}u(x_{-w_0})u(\hatx_{w_0})f_u)\\
&=\sum_{w'}\sum_{u}\amp_{w,u}\bmp_{u,v}\frac{u(x_{-w_0})u(\hatx_{w_0})}{u(x_{-w_0})u(\hatx_{w_0})}f_{uw^{'-1}}\\
&=\sum_{w'}(\sum_u\amp_{w,u}\bmp_{u,v})f_{u(w')^{-1}}=\de_{w,v}\sum_{w'}f_{u(w')^{-1}}=\de_{w,v}\unit.
\end{flalign*}
\end{proof}

The following is the algebraic model of the duality between stable bases for the positive and negative chambers, see Remark \ref{rem:afterdef}.(2).
\begin{theorem}\label{thm:main}
\begin{enumerate}
\item Notations as above, we have 
\begin{gather*}
\hatY_\Pi\bullet \left(\St^+_w\cdot \St^-_u\right)=\de_{w,u}q_{w_0u}^{-1}\unit.
\end{gather*}
\item This duality coincide with the duality in Remark \ref{rem:afterdef}.(2).
\end{enumerate} 
\end{theorem}
\begin{proof}
\begin{flalign*} We have
\hatY_\Pi\bullet(\St^+_w\cdot \St^-_u) = \hatY_\Pi\bullet \left((\taup_{w^{-1}}\bullet \pt_e)\cdot [(\taum_{w_0u})^{-1}\bullet \pt_{w_0}]\right)
\overset{\sharp}=\hatY_\Pi\bullet \left(\pt_e\cdot [\taum_w\bullet (\taum_{w_0u})^{-1}\bullet \pt_{w_0}]\right),
\end{flalign*}
where $\sharp$ follows from Lemma \ref{lem:adjoint}.  Since $\pt_e=x_{-w_0}f_e$ and $f_u\cdot f_v=\de_{u, v}f_u$, it suffices to look at the term involving $f_e$ in $\taum_w(\taum_{w_0u})^{-1}\bullet\pt_{w_0}$. Furthermore, since $\de_u\bullet f_v=f_{vu^{-1}}$ (see \eqref{eq:bullet}) and $\pt_{w_0}=x_{-w_0}f_{w_0}$,  it suffices to look at the term involving $\de_{w_0}$ inside $\taum_w(\taum_{w_0u})^{-1}$. Lastly, from Lemma \ref{lem:basis} we know that $\taum_w=\sum_{v\le w}a^-_{w,v}\de_v$, so it reduces to look at the term $\taum_{w_0}$ inside $\taum_w(\taum_{w_0u})^{-1}$, which is $\de_{w,u}q_{w_0u}^{-1}\taum_{w_0}$ by Lemma \ref{lem:dualcoeff}. So we have
\begin{eqnarray*}
\hatY_\Pi\bullet(\St^+_w\cdot \St^-_u)&=&\hatY_\Pi\bullet (\pt_e\cdot (\de_{w,u}q_{w_0u}^{-1}\taum_{w_0}\bullet \pt_{w_0}))\\
 &\overset{\sharp_1}=&\de_{w,u}q_{w_0u}^{-1}\hatY_\Pi\bullet \left([x_{-w_0}f_e]\cdot [\frac{\hatx_{w_0}}{x_{-w_0}}\de_{w_0}\bullet (x_{-w_0}f_{w_0})]\right)\\
&=&\de_{w,u}q_{w_0u}^{-1}\sum_{v\in W}\de_v\frac{1}{x_{-w_0}\hatx_{w_0}}\bullet \left([x_{-w_0}f_e]\cdot [\hatx_{w_0}f_e]\right)\\
&=&\de_{w,u}q_{w_0u}^{-1}\sum_{v\in W}\de_v\frac{1}{x_{-w_0}\hatx_{w_0}}\bullet\left(x_{-w_0}\hatx_{w_0}f_e)\right)=\de_{w,u}q_{w_0u}^{-1}\sum_{v\in W}f_{v^{-1}}=\de_{w,u}q_{w_0u}^{-1}\unit.
\end{eqnarray*}
Here $\sharp_1$ follows from Lemma \ref{lem:basis} and the other identities follow from \eqref{eq:bullet}.
\end{proof}
\begin{remark}Recall from \cite{KK90, CZZ1} that there is an element $X_i=Y_i-1$ (called the divided difference element) inside $Q_W$, and one can define $X_w$ correspondingly. The method in Theorem \ref{thm:main} works also for the $X_w$ and $Y_w$ operators. More precisely, by using analogue of Lemma \ref{lem:dualcoeff} and \cite[Lemma 7.1]{CZZ2}, we can similarly show that 
\begin{gather*}Y_\Pi\bullet \left([X_{w^{-1}}\bullet \pt_e]\cdot [Y_{u^{-1}w_0}\bullet (x_{w_0}f_{w_0})]\right)=\de_{w,u}\unit.\\
Y_\Pi\bullet \left([Y_{w^{-1}}\bullet \pt_e]\cdot [X_{u^{-1}w_0}\bullet (x_{w_0}f_{w_0})]\right)=\de_{w,u}\unit.
 \end{gather*}
Note that $Y_{w^{-1}}\bullet \pt_{e}$ gives the Schubert class corresponding to $w$.  This proof is different from the one given in \cite{LZ14}, moreover, it can be easily generalized to the connective K-theory case. 
\end{remark}

\section{The restriction formula}\label{sec:root}

In this section we use the root polynomials to study the coefficients $\bpm_{w,v}$ introduced in Lemma \ref{lem:basis}. Our method generalizes the formulation in \cite{LZ14}. In particular, this allows us to avoid the direct calculations in checking the dependence of root polynomials on reduced sequences.
\subsection{The evaluation map}
Throughout this section, we denote $Q^x\cong Q$ if variables of $Q$ are denoted by $x_\la=1-e^{-\la}$. Variables of $Q^y\cong Q$ will be denoted by $y_\la$. Let $\hatQ=Q^y\otimes_{R}Q^x$, and consider the ring $\hatQ_W:=Q^y\otimes_{R }Q^x_W$ where   elements of $Q^y$ commute with elements of $Q^x_W$. The free $\hatQ$-module $\hatQ_W$  has basis $\{\de_w^x\}_{w\in W}$. We define a ring homomorphism 
\[\ev: \hat Q=Q^y\otimes_{R}Q^x\to Q^x, \quad y_\la\otimes x_\mu\mapsto x_\la x_\mu,\]which induces a left $\hatQ$-module structure on $Q^x_W$.  The map $\ev$ also induces a left $\hatQ$-module homomorphism $\ev: \hatQ_W\cong Q^y\otimes_{R}Q^x_W\to Q^x_W$. It is easy to check that 
\begin{equation}\label{eq:ev}
\ev(y\hat{z}z)=\ev(y)\ev(\hat z)\ev(z), \quad y\in Q^y, \hat z\in \hatQ_W, z\in Q^x_W.
\end{equation}

Given a set $\{a_\al, b_\al\}_{\al\in \Sigma}\subset Q$, denote $a_i=a_{\al_i},b_i=b_{\al_i}$, and define $a_w,b_w$ as the corresponding products of $a_\al$ and $b_\al$, similar as in \eqref{eq:1}. We will use $a^y_\la, b^y_\la$ (resp. $a^x_\la, b^x_\la$) when they are considered as inside $Q^y$ (resp. $Q^x$). 

For each simple root $\alpha_i$, we consider $\sigma_i=a_i\de_i+b_i\in Q_w$. They satisfy the braid relations, hence $\sigma_v$ is well-defined for any $v$. When considering  $\sigma_v$ as an element in $Q_W^x$, we denote it by $\sigma_v^x$.

\subsection{The root polynomials}
\begin{definition}For any $w=s_{i_1}\cdots s_{i_l}$, denote $\beta_j=s_{i_1}\cdots s_{i_{j-1}}\al_j$, and define the root polynomial
\begin{equation}\label{eq:Rdef}
\cR_{w}^\sigma=\prod_{j=1}^lh^\sigma_{i_j}(\beta_j)\in \hatQ_W, \quad \text{where }h^\sigma_i(\be)=\sigma^x_i-b^y_\be\in \hatQ_W.
\end{equation}
\end{definition}
Denote 
\begin{equation}\label{eq:rootviaT}
\cR_{w}^\sigma=\sum_{v\le w}K^\sigma_{v,w}\sigma^x_{v}\in \hatQ_W, \quad K^{\sigma}_{v,w}\in Q^y\otimes Q^x.
\end{equation}
Since $\sigma_i^x$ satisfy the braid relations, we have $K^\sigma_{v,w}\in Q^y$.
The following theorem generalizes \cite[Lemma 3.3]{LZ14}.
\begin{theorem}\label{thm:root}\begin{enumerate}
\item $\ev(\cR_w^\sigma)=a_w^x\de_w^x$.
\item Writing $\de_w^x=\sum_{v}b^\sigma_{w,v}\sigma^x_v$, then $a_w^xb^\sigma_{w,v}=\ev(K^\sigma_{v,w})$. In particular, $K^\sigma_{v,w}$ and hence $\cR^\sigma_w$ do not depend on the choice of the reduced sequence of $w$.
\end{enumerate}
\end{theorem}
\begin{proof}(1). We use induction on $\ell(w)$. If $w=s_i$, 
\[
\ev(\cR^\sigma_i)=\ev(\sigma^x_i-b_i^y)=\sigma^x_i-b^x_i=a_i^x\de_i^x.
\]
 Assume the conclusion holds for all $v$ such that $\ell(v)\le k$, i.e.,  $\ev(\cR_{{v}}^\sigma)=a_v^x\de^x_v$. Suppose that $w=vs_i$ with $\ell(v)=k=\ell(w)-1$. Then $\Sigma_w=\Sigma_v\sqcup \{v(\al_i)\}$, and we have
\begin{gather*}
\ev(\cR^\sigma_w)=\ev[\cR_v^\sigma \cdot (\sigma_i^x-b^y_{v(\al_i)})]\overset{\sharp_1}=\ev[\cR_v^\sigma\sigma_i^x-b^y_{v(\al_i)}\cR_v^\sigma]\\
\overset{\sharp_2}=a_v^x\de^x_v\sigma^x_i-b^x_{v(\al_i)}a_v^x\de_v^x=a_v\de_v(\sigma^x_i-b^x_i)=a_v^x\de_v^xa_i^x\de_i^x=a^x_w\de^x_w.
\end{gather*}
Here identity $\sharp_1$ follows since $b^y_\be\in Q^y$ commutes with elements of $\hatQ_W$, and $\sharp_2$ follows from \eqref{eq:ev}.

(2). Applying $\ev$ on \eqref{eq:Rdef}, we have \[
a_w^x\de_w^x=\ev(\cR_w^\sigma)=\sum_v\ev(K^\sigma_{v,w})\ev(\sigma^x_v)=\sum_v\ev(K^\sigma_{v,w})\sigma^x_v.
\]
So $\frac{1}{a_w^x}\ev(K^\sigma_{v,w})=b^\sigma_{w,v}$. Since $K^\sigma_{v,w}\in Q^y$ and $\ev$ only changes the $y$-variables into $x$-variables, we see that $K^\sigma_{v,w}=\ev(K^\sigma_{v,w})=a_w^xb^\sigma_{w,v}$ if we identify the $x$ and $y$-variables. Therefore, $K_{v,w}^\sigma$ and hence $\cR^\sigma_w$ do not depend on the choice of reduced decompositions. 
\end{proof}

\subsection{Root polynomials of Demazure-Lusztig elements}\label{subsec:rootDL}
We now apply the root polynomial construction to the $\taupm_i$ operators. We have 
\[
\cR_{w}^\taupm=\prod_{j=1}^lh^\taupm_{i_j}(\beta_j), \quad \text{where }h^\taupm_i(\be)=\taupm_i^x-\frac{q-1}{y_{-\be}}.
\]
Expanding in terms of $\taupm^x_v$, we write
$
\cR_w^\taupm=\sum_{v}K^\taupm_{v,w}\taupm^x_v.
$
Note that  $\taup_i$ and $\taum_i$ satisfy the same quadratic relations, and that their corresponding root polynomials $h^\taup_i(\be)$ and $h^\taum_i(\be)$ have the same form. Hence,  $K^\taup_{v,w}=K^\taum_{v,w}\in Q^y$, which will be denoted by $K^\tau_{v,w}$. Applying Theorem \ref{thm:root} to $\cR_w^\taupm$, we get the following.
\begin{theorem}\label{thm:rootcoeff}
\begin{enumerate}
\item $\ev(\cR_{w}^{\taup})=\frac{\tilx_w}{x_{w}}\de_w^x, ~\ev(\cR_{w}^\taum)=\frac{\hatx_w}{x_{-w}}\de^x_w$. 
\item $\frac{\tilx_w}{x_{w}}\bp_{w,v}=K^\tau_{v,w}=\frac{\hatx_w}{x_{-w}}\bm_{w,v}$. 
\end{enumerate}
\end{theorem}

\begin{remark}
The formal root polynomials can be defined similarly for the formal affine Hecke algebra  \cite{ZZ14}. They do not depend on the choice of reduced sequence for hyperbolic formal group law. Moreover, restricting to the connective $K$-theory, one  gets a uniform treatment of the restriction formulas of K-theoretic stable bases in this paper and that of  cohomological stable bases in \cite{Su15}.
\end{remark}

\subsection{Restriction formula via root polynomials}
The following theorem gives the restriction formulas of stable bases of $T^*G/B$:
\begin{theorem}\label{thm:1}
\begin{enumerate}
\item  $\Stp_w=\sum_{v\le w}v(\ap_{w^{-1}, v^{-1}})x_{-w_0}f_v.$
\item $q_{w_0w}\Stm_w=\hatx_{w_0}\taum_w^*=\sum_{v\ge w}\hatx_{w_0}\bm_{v,w}f_v=\sum_{v\ge w}\frac{\hatx_{w_0} x_{-v}}{\hatx_v}K^\tau_{w,v}f_v=\sum_{v\ge w}v(\hatx_{v^{-1}w_0})x_{-v}K^\tau_{w,v}f_v.$
\end{enumerate}
\end{theorem}
\begin{proof}
(1). This follows from definition and Lemma \ref{lem:basis}.

(2). Via the pairing defined in Lemma \ref{lem:dualT}, $\hatx_{w_0}\taum_w^*$ is dual to $\St^+_u$.  According to Theorem \ref{thm:main}, $q_{w_0w}\St^-_w$ is also dual to $\St^+_u$. Hence, $q_{w_0w}\St^-_w=\hatx_{w_0}\taum^*_w$.  The second equality in the theorem follows from the definition of $\taum_u^*$; the third equality follows from Theorem \ref{thm:rootcoeff}; the last identity follows the identities 
\[
\Sigma^+\backslash (v\Sigma^-\cap \Sigma^+)=\Sigma^+\backslash v\Sigma^-=v\left(v^{-1}\Sigma^+\backslash \Sigma^-\right)=v\left(v^{-1}\Sigma^+\cap \Sigma^+\right)=v\left(v^{-1}w_0\Sigma^-\cap \Sigma^+\right)=v(\Sigma_{v^{-1}w_0}). 
\]
\end{proof}
\begin{example}From Theorem \ref{thm:1} and Lemma \ref{lem:basis} we have 
\begin{gather*}\Stp_w|w=w(\ap_{w^{-1}, w^{-1}})x_{-w_0}=(\prod_{\al<0, w^{-1}\al>0}\tilx_\al)\cdot(\prod_{\be>0, w^{-1}\be>0}x_{-\be})=[\prod_{\al<0, w^{-1}\al>0}(q-e^\al)]\cdot [\prod_{\be>0, w^{-1}\be>0}(1-e^{\be})],\\
q_{w_0w}\Stm_w|_w=\hatx_{w_0}\frac{x_{-w}}{\hatx_w}=(\prod_{\al>0, w^{-1}\al>0}\hatx_\al)\cdot (\prod_{\be>0, w^{-1}\be<0}x_{-\be})=[\prod_{\al>0, w^{-1}\al>0}(1-qe^{-\al})]\cdot [\prod_{\be>0, w^{-1}\be<0}(1-e^\be)].
\end{gather*}
\end{example}

\section{Stable bases of partial flag varieties}\label{sec:geoPJ}
Let $J$ be a subset in the set of  simple roots, let $G/P_J$ be the partial flag variety corresponding to $J$. In this section, we consider the stable bases of $K_T(T^*G/P_J)$. The main result of this section (Theorem \ref{thm:geoPJ}) says that  such bases coincide with the image of the stable bases of $K_T(T^*G/B)$ via the map \eqref{eq:p2p1}.
We then give an algebraic formula for the stable bases in this case.

\subsection{The definition of stable basis}
The $A$-fixed points of $T^*G/P_J$ under the maximal torus $A$ are indexed by the left cosets $W/W_J$, or by $W^J$. Moreover, $v\Sigma_J^+\subset \Sigma^+$ for $v\in W^J$. As in Section \ref{sec:geodef}, we can define chambers, partial orders on the fixed points, leaves, slopes and polarizations in the setting of $T^*G/P_J$. The group  $\Pic(T^*G/P_J)$ is isomorphic to the lattice $\{\lambda\in \Lambda| (\lambda,\alpha^\vee)=0 \text{ for any } \alpha\in J\}$. 

We use the following theorem as the definition of stable bases of $T^*G/P_J$:

\begin{theorem}\label{thm:geostablepartial} \cite[\S 9.1]{OK15}  For any chamber $\fC$, any polarization $T^{\frac{1}{2}}$ of $T^*G/P_J$, and any rational line bundle $\calL$, there exists a unique map of $S$-modules
\[
\stab^J_{\mathfrak{C},T^{\frac{1}{2}},\calL}:K_T((T^*G/P_J)^A)\rightarrow K_T(T^*G/P_J),
\]
such that for any $w\in W^J$, $\Gamma=\stab^J_{\mathfrak{C},T^{\frac{1}{2}},\calL}(w)$ satisfies:
\begin{enumerate}
\item (\textit{support}) $\supp \Gamma\subset \Slope_{\mathfrak{C}}(w)$;
\item (\textit{normalization}) $\Gamma|_w=(-1)^{\rank N_{w,+}^{\frac{1}{2}}}\left(\frac{\det N_{w,-}}{\det N_w^{\frac{1}{2}}}\right)^{\frac{1}{2}}\calO_{\Leaf_\mathfrak{C}(w)}|_w$;
\item (\textit{degree}) $\deg_A\left(\Gamma|_v\otimes \calL|_w\right)\subseteq \deg_A\left((\stab^J_{\mathfrak{C},T^{\frac{1}{2}},\calL}(v)\otimes\calL)|_v\right)$, for any $v\in W^J$ and $v\prec_{\mathfrak{C}} w$,
\end{enumerate}
where $w$ in $\stab^J_{\mathfrak{C},T^{\frac{1}{2}},\calL}(w)$ is the unit in $K_T^*(w)$.
\end{theorem}

Then the stable basis for $K_T(T^*G/P_J)_{\textit{loc}}$ is $\{\stab^J_{\mathfrak{C},T^{\frac{1}{2}},\calL}(w)|w\in W^J\}$. And we have the following duality property \cite[Proposition 1]{OS16}:
\begin{equation}\label{duality for partial}
\left(\stab^J_{\mathfrak{C},T^{\frac{1}{2}},\calL}(v),~\stab^J_{\mathfrak{-C},T_{\text{opp}}^{\frac{1}{2}},\calL^{-1}}(w)\right)=\delta_{v,w}.
\end{equation}
Let $\leq^J$ denote the Bruhat order on $W^J$, i.e., for any $v,w\in W^J$, $v\leq^J w$ if $BvP_J/P_J\subset\overline{BwP_J/P_J}$. Similarly to Lemma \ref{lem:charstab}, we have
\begin{lemma}\label{lem:charpartialstab}
For any $v,w\in W^J$, we have
\begin{enumerate}
\item 
$\stab^J_{-,T^*G/P_J,\calL}(v)|_w=0$, unless $v\leq^J w$.
\item 
$\stab^J_{-,T^*G/P_J,\calL}(v)|_v=q_v^{\frac{1}{2}}\prod\limits_{\beta\in \Sigma^+\setminus \Sigma^+_J, -v\beta\in \Sigma^+}(1-e^{-v\beta})\prod\limits_{\beta\in \Sigma^+\setminus \Sigma^+_J, v\beta\in \Sigma^+}(1-qe^{-v\beta})$.
\end{enumerate}
\end{lemma}
\begin{proof}
(1) follows from the the support condition. 

(2). Since we choose the negative chamber $-$, we have
\begin{flalign*}A-\text{weights in } N_{v,+}&=\{e^{-v\beta}|\beta\in \Sigma^+\setminus \Sigma^+_J,v\beta>0\}\cup \{q^{-1}e^{v\beta}|\beta\in \Sigma^+\setminus \Sigma^+_J,v\beta<0\},\\
A-\text{weights in } N_{v,-}&=\{e^{-v\beta}|\beta\in \Sigma^+\setminus \Sigma^+_J,v\beta<0\}\cup \{q^{-1}e^{v\beta}|\beta\in \Sigma^+\setminus \Sigma^+_J,v\beta>0\},\\
A-\text{weights in } N_v^{\frac{1}{2}}&=\{q^{-1}e^{v\beta}|\beta\in \Sigma^+\setminus \Sigma^+_J\}.\end{flalign*}
And $\Leaf(w)=T_{B^-vP_J/P_J}^*(G/P_J)$, where $P_J$ is the corresponding parabolic subgroup. Therefore
\begin{flalign*}
\stab_-(v)|_v&=(-1)^{\rank N_{v,+}^{\frac{1}{2}}}\left(\frac{\det N_{v,-}}{\det N_v^{\frac{1}{2}}}\right)^{\frac{1}{2}}\calO_{\Leaf_\mathfrak{C}(v)}|_v\\
&=(-1)^{\ell(v)}\left(\frac{\prod\limits_{\beta\in \Sigma^+\setminus \Sigma^+_J,v\beta<0}e^{-v\beta}\prod\limits_{\beta\in \Sigma^+\setminus \Sigma^+_J,v\beta>0}q^{-1}e^{v\beta}}{\prod\limits_{\beta\in \Sigma^+\setminus \Sigma^+_J}q^{-1}e^{v\beta}}\right)^{\frac{1}{2}}\prod\limits_{\beta\in \Sigma^+\setminus \Sigma^+_J,v\beta<0}(1-e^{v\beta})\prod\limits_{\beta\in \Sigma^+\setminus \Sigma^+_J,v\beta>0}(1-qe^{-v\beta})\\
&=q_v^{\frac{1}{2}}\prod\limits_{\beta\in \Sigma^+\setminus \Sigma^+_J,v\beta<0}(1-e^{-v\beta})\prod\limits\limits_{\beta\in \Sigma^+\setminus \Sigma^+_J,v\beta>0}(1-qe^{-v\beta}).
\end{flalign*}
\end{proof}

Therefore, as in Section \ref{sec:rigidity}, we have
\begin{equation}\label{deg cond for partial}
max_\xi(\stab^J_{-,T^*G/P_J,\calL}(v)|_v)=(\xi, \sum_{\beta\in \Sigma^+\setminus \Sigma^+_J,v\beta<0} -v\beta),\quad min_\xi(\stab^J_{-,T^*G/P_J,\calL}(v)|_v)=(\xi, \sum_{\beta\in \Sigma^+\setminus \Sigma^+_J,v\beta>0} -v\beta).
\end{equation}

 We have a projection $\pi:G/B\to G/P_J$ and  a Lagrangian correspondence  $G/B\times_{G/P_J}T^*G/P_J$ in $T^*G/B\times T^*G/P_J$: 
\[\xymatrix{ T^*G/B &
G/B\times_{G/P_J}T^*G/P_J \ar[l]_-{p_1} \ar[r]^-{p_2} & T^*G/P_J }. \]
Therefore, we have the following map:
\begin{equation}\label{eq:p2p1}
p_{2*}p_1^*:K_T(T^*G/B)\rightarrow K_T(T^*G/P_J).
\end{equation}
Recall $\calL=\calO(\lambda)\in \Pic(T^*G/B)\otimes_{\bbZ}\bbQ$, where $\lambda$ lies in the fundamental alcove and $\lambda$ is sufficiently near $0$. Let $\calL_J:=\calO(\lambda-\sum_{\alpha\in J}(\lambda,\alpha^\vee)\varpi_\alpha)\in \Pic(T^*G/P_J)\otimes_{\bbZ}\bbQ$, where $\varpi_\alpha$ is the fundamental weight associated to the simple root $\alpha$. 

For any $v\in W^J$, denote
\[\stab^J_+(v)=\stab^J_{+,T(G/P_J),\calL_J^{-1}}(v), \quad \stab^J_-(v)=
\stab^J_{-,T^*(G/P_J),\calL_J}(v).\] The image of the stable bases under the map (\ref{eq:p2p1}) is given as follows.

\begin{theorem}\label{thm:geoPJ}
For any $v\in W^J$, we have
\[p_{2*}p_1^*(\stab_+(v))=\stab^J_+(v),\] 
and 
\[p_{2*}p_1^*(\stab_-(v))=\stab^J_-(v).\]   
\end{theorem}
We use the rigidity technique from Section~\ref{sec:rigidity}.
\begin{proof}
Thanks to the duality property (\ref{duality for partial}), the first identity is equivalent to
\[\left(p_{2*}p_1^*(\stab_+(v)), \stab^J_-(u)\right)=\delta_{v,u},\]
for any $u\in W^J$, which we now prove. 

By the support condition of stable basis, $\left(p_{2*}p_1^*(\stab_+(v)), \stab^J_-(u)\right)$ is a proper intersection number, hence an element in $K_T(\pt)$, i.e.,  a Laurent polynomial. 
Localizing to the $T$-fixed points, we get
\begin{align}\label{loc sum}
\left(p_{2*}p_1^*(\stab_+(v)), \stab^J_-(u)\right)=&\left(p_1^*(\stab_+(v)), p_2^*(\stab^J_-(u))\right)\\
=&\sum_{w\in W^J, y\in wW_J,y\leq v,u\leq^J w}\frac{\stab_+(v)|_y\stab^J_-(u)|_w}{\bigwedge\nolimits^\bullet T_{(y,w)}(G/B\times_{G/P_J}T^*G/P_J)}.
\end{align}
Note that 
\begin{equation}\label{deno weights}
\bigwedge\nolimits^\bullet T_{(y,w)}(G/B\times_{G/P_J}T^*G/P_J)=\prod_{\beta>0}(1-e^{y\beta})\prod_{\beta\in \Sigma^+\setminus \Sigma_J^+}(1-qe^{-w\beta}).
\end{equation}

Let $\xi$ be as in Lemma \ref{lem:rigidity}. In particular, 
\[max_\xi(\bigwedge\nolimits^\bullet T_{(y,w)}(G/B\times_{G/P_J}T^*G/P_J))=\sum_{\beta>0,y\beta>0}(\xi, y\beta)+\sum_{\beta\in \Sigma^+\setminus \Sigma_J^+, w\beta<0}(\xi, -w\beta).\]
By the third conditions of Theorem \ref{thm:geostable} and Theorem \ref{thm:geostablepartial}, and Equations (\ref{deg con}) and (\ref{deg cond for partial}), we have
\[max_\xi(\stab_+(v)|_y\stab^J_-(u)|_w)\leq \sum_{\beta>0,y\beta>0}(\xi, y\beta)+\sum_{\beta\in \Sigma^+\setminus\Sigma^+_J, w\beta<0}(\xi, -w\beta)+(\xi, \calL|_v-\calL|_y+\calL_J|_w-\calL_J|_u). \]
Since $v\geq y$ and $w\geq^J u$, Lemma \ref{lem:Bruhatcompare} shows that
\[(\xi, \calL|_v-\calL|_y+\calL_J|_w-\calL_J|_u)\leq 0,\]
with strict inequality if $u\neq v$.

Now we analyze separately the following two cases:  $u\neq v$ and $u=v$. In the case when $u\neq v$, we have
\[\lim_{t\rightarrow\infty}\left(p_{2*}p_1^*(\stab_+(v)), \stab^J_-(u)\right)(t\xi)=0.\]
To analyze the limit as $t$ goes to $-\infty$, we may assume $\lambda$ sufficiently small so that 
\[(\xi, \calL|_v-\calL|_y+\calL_J|_w-\calL_J|_u)> -1.\] Here $\calL=\calO(\lambda)$. 
Under this condition, keeping in mind that $u\neq v$, we have
\[\lim_{t\rightarrow -\infty}\left(p_{2*}p_1^*(\stab_+(v)), \stab^J_-(u)\right)(t\xi) \text{ is bounded.}\]
Hence, by Lemma \ref{lem:rigidity}, 
\[\left(p_{2*}p_1^*(\stab_+(v)), \stab^J_-(u)\right)=0.\]

In the case when $u=v\in W^J$, we have
\[\{(y,w)\mid y\in W,w\in W^J, y\in wW_J,y\leq v,u\leq^J w\}=\{(u,v)\}. \]
Indeed, any $(y,w)$ in the left hand side satisfies
\[\ell(y)\leq \ell(v)=\ell(u)\leq \ell(w)\leq \ell(y),\]
hence also in the right hand side. Therefore, the summation \eqref{loc sum} has only one term. Using Lemma \ref{lem:charstab}, Lemma \ref{lem:charpartialstab}, Equation (\ref{deno weights}), keeping in mind  that $v\Sigma_J^+\subset\Sigma^+$, we get
\begin{align*}
\left(p_{2*}p_1^*(\stab_+(v)), \stab^J_-(v)\right)
=\frac{\stab_+(v)|_v\stab^J_-(v)|_v}{\bigwedge\nolimits^\bullet T_{(v,u)}(G/B\times_{G/P_J}T^*G/P_J)}=1.
\end{align*}

This proves the identity $p_{2*}p_1^*(\stab_+(v))=\stab^J_+(v).$
The identity
$p_{2*}p_1^*(\stab_-(v))=\stab^J_-(v)$ is proved using the same argument.
\end{proof}

\subsection{More on the twisted group algebra}
Let $\pi:G/B\to G/P_J$ be the canonical map. By \cite[Lemma 10.12]{CZZ2}, $Y_J\in \DF$. Indeed, $Y_{\{\al_i\}}=Y_{i}$. It follows from Kostant-Kumar (or see \cite[Theorem 8.2 and Corollary 8.7]{CZZ3} for more details) that we have commutative diagrams 
\[
\xymatrix{
K_T(G/B)\ar[d]^\sim\ar[r]^{\pi^*\pi_*} & K_T(G/B)\ar[d]^\sim  & & K_T(G/P_J)\ar[r]\ar[d]^\sim & K_T(\bbC)\ar[d]^\sim\\
 \DFd\ar[r]^{Y_J\bullet\_}&\DFd && (\DFd)^{W_J}\ar[r]^{Y_{\Pi/J}\bullet\_} & (\DFd)^W.}
\]
Here the top horizontal map in the second diagram is induced by the structure map $G/P_J\to \Spec(\bbC)$. 

Let $p:T^*G/B\to G/B$ and $p_J:T^*G/P_J\to G/P_J$ be the canonical projections. Then
\[Q_W^*\cong K_T(G/B)\otimes_SQ\overset{p^*}{\underset\sim\longrightarrow}K_T(T^*G/B)\otimes_SQ, \quad (Q_W^*)^{W_J}\cong K_T(G/P_J)\otimes_SQ\overset{p_J^*}{\underset\sim\longrightarrow}K_T(T^*G/P_J)\otimes_SQ. \]
 Moreover, $p^*=\hatx_{w_0}\bullet\_$ and $p_J^*=\frac{\hatx_{w_0}}{\hatx_{w_0^J}}\bullet\_$.

 Via these isomorphisms,  $\hatY_\Pi\bullet\_$ and $\hatY_{\Pi/J}\bullet \_$ coincide with the following composites, respectively: \[K_T(T^*G/B)\underset{\sim}{\overset{(p^*)^{-1}}\longrightarrow} K_T(G/B)\longrightarrow K_T(\bbC), \quad K_T(T^*G/P_J)\underset{\sim}{\overset{(p_J^*)^{-1}}\longrightarrow} K_T(G/P_J)\longrightarrow K_T(\bbC).\]

Concerning the map $p_{2*}p_1^*$ in \eqref{eq:p2p1}, we have
\begin{lemma}\label{lem:parapush}\begin{enumerate}
\item We have a commutative diagram
\[
\xymatrix{K_T(G/B)\ar[r]^{p^*}\ar[d]^{\pi_*} & K_T(T^*G/B)\ar[d]^{p_{2*}p_1^*}\\
K_T(G/P_J) \ar[r]^{p_J^*}&K_T(T^*G/P_J).}
\] That is,  $p_J^*\circ \pi_*=(p_{2*} p_1^*)\circ p^*$.
\item Via the $\bullet$-action of $Q_W$ on $Q_W^*$, we have $p_{2*}p_1*=\hatY_{J}$. 
\end{enumerate}
\end{lemma}
\begin{proof}
(1). This follows from the proper base change property of K-theory. 

(2). We know that $\pi_*=Y_J\bullet\_$, $p_J^*=\frac{\hatx_{w_0}}{\hatx_{w_0^J}}\bullet\_$ and $p^*=\hatx_{w_0}\bullet\_$. So 
\[
p_{2*}p_1^*=p_J^*\circ \pi_*\circ (p^*)^{-1}=\frac{\hatx_{w_0}}{\hatx_{w_0^J}}\sum_{w\in W_J}\de_w\frac{1}{x_{w_0^J}}\frac{1}{\hatx_{w_0}}\overset{\sharp}=\sum_{w\in W_J}\de_w\frac{1}{x_{w_0^J}\hatx_{w_0^J}}=\hatY_J.
\]
Here to show $\sharp$, note that in the case $w\in W_J$, we have $w(\Sigma^+\backslash\Sigma^+_J)=\Sigma^+\backslash\Sigma^+_J$ and consequently $w(\frac{\hatx_{w_0}}{\hatx_{w_0^J}})=\frac{\hatx_{w_0}}{\hatx_{w_0^J}}$.
\end{proof}

\subsection{The algebraic descriptions}
\begin{definition}\label{def:partialstab}Let $w\in W^J$. We define elements in  $Q_W^*$ by 
\[
\St^{+,J}_w:=\hatY_J\bullet \Stp_w=\hatY_J\bullet (\taup_{w^{-1}}\bullet \pt_e), \quad \St^{-,J}_w:=\hatY_J\bullet \Stm_w=\hatY_J\bullet ((\taum_{w_0w})^{-1}\bullet \pt_{w_0}).
\]
\end{definition}

\begin{theorem}\label{thm:resPJ}For any $w\in W^J$, denote $g_w=\sum_{u\in W}f_{wu}$, then we have
\[
\St^{+,J}_w=\sum_{v\le w, v\in W^J}x_{w_0}v(\frac{a^+_{w^{-1}, v^{-1}}}{x_{-w_0^J}\hatx_{w_0^J}})g_v, \quad q_{w_0v}\St^{-,J}_v=\sum_{v\ge w, v\in W^J}x_{-v}K^\tau_{w,v}v(\frac{\hatx_{v^{-1}w_0}}{x_{-w_0^J}\hatx_{w_0^J}})g_v=\sum_{v\ge w, v\in W^J}\frac{\hatx_{w_0}b^-_{v,w}}{v(x_{-w_0^J}\hatx_{w_0^J})}g_v. 
\]
\end{theorem}
\begin{proof}It follows from the definition of $\hatY_J$ in \eqref{eq:YJ}, the identities in \eqref{eq:bullet}, and Theorem \ref{thm:1}.
\end{proof}

The following give a purely algebraic description of the geometrically defined stable bases.
\begin{corollary}\label{cor:mainpartial}
We have 
\[\stab^J_{+}(w)=q_w^{-\frac12}\St^{+,J}_{w}, \qquad \stab^J_{-}(w)=q_{w_0}q_w^{-\frac12}\St^{-,J}_w.\]
\end{corollary}
\begin{proof}IAccording to Lemma \ref{lem:parapush}.(1), $p_{2*}p_1^*=\hatY_J\bullet \_$. This corollary now follows from Corollary \ref{cor:alggeo}. 
\end{proof}

\begin{corollary}\label{cor:partialdual}
 $\hatY_{\Pi/J}\bullet (\St^{+,J}_w\cdot \St^{-,J}_v)=\de_{w,v}q_{w_0v}^{-1}\unit.$
\end{corollary}
\begin{proof} Let $\pi^J:G/P_J\to \bbC$  be the structure map. The composition 
\[K_T(T^*G/P_J)\overset{(p_J^*)^{-1}}\longrightarrow K_T(G/P_J)\overset{\pi_{*}^J}\to K_T(\bbC)\] is given by the formula $\hatY_{\Pi/J}\bullet \_:(Q_W^*)^{W_J}\to (Q_W^*)^{W}$.  By \eqref{duality for partial} and Corollary  \ref{cor:mainpartial}, we have 	
\[
\de_{w,v}\unit=\hatY_{\Pi/J}\bullet (\stab^J_{+}(w)\cdot \stab^J_{-}(v))=\hatY_{\Pi/J}\bullet (q_w^{-\frac12}\St^{+,J}_w\cdot q_{w_0}q_{v}^{-\frac{1}{2}}\St^{-,J}_v),
\]
the conclusion then follows. 
\end{proof}
\begin{corollary} Expressing $\taum_w\hatY_J(\taum_{w_0v})^{-1}$ in terms of $\taum_u, u\in W$, the coefficient in front of $\taum_{w_0}$ is equal to $\de_{w,v}q_{w_0v}^{-1}$. 
\end{corollary}
\begin{proof}
 Denote this coefficient by $c$. By definition of $\hatY_J$ we know that $\St^{\pm, J}_w\in (Q_W^*)^{W_J}$. By Corollary \ref{cor:partialdual} we have 
\begin{gather*}
q_{w_0v}^{-1}\de_{w,v}\unit\overset{\sharp_1}=\hatY_{\Pi/J}\bullet [(\hatY_J\bullet \Stp_w)\cdot (\hatY_J\bullet \Stm_v)]\overset{\sharp_2}=\hatY_{\Pi/J}\bullet [\hatY_J\bullet (\Stp_w \cdot (\hatY_J\bullet \Stm_v))]\\
\overset{\sharp_3}=\hatY_\Pi\bullet[\Stp_w\cdot (\hatY_J\bullet \Stm_v)]=\hatY_\Pi\bullet[(\taup_{w^{-1}}\bullet \pt_e)\cdot (\hatY_J(\taum_{w_0v})^{-1}\bullet \pt_{w_0})]\overset{\sharp_4}=\hatY_\Pi\bullet[\pt_e\cdot (\taum_w\hatY_J(\taum_{w_0v})^{-1}\bullet \pt_{w_0})]\\
\overset{\sharp_5}=\hatY_\Pi\bullet[(x_{-w_0}f_e)\cdot (c\taum_{w_0}\bullet x_{-w_0}f_{w_0})]\overset{\sharp_6}=c\frac{x_{-w_0}\hatx_{-w_0}}{x_{-w_0}\hatx_{w_0}}\unit=c\unit. 
\end{gather*}
Here $\sharp_1$ follows from the definition of $\St^{\pm, J}_w$;  $\sharp_2$ follows from the projection formula \eqref{eq:projfor}; $\sharp_3$ follows from \eqref{eq:comp}; $\sharp_4$ follows from Lemma \ref{lem:adjoint}; $\sharp_5$ follows from similar idea in the proof of Theorem \ref{thm:main}; $\sharp_6$ follows from \eqref{eq:bullet}.  We then have $c=q_{w_0v	}^{-1}$. 
\end{proof}
This corollary is the parabolic version of Lemma \ref{lem:dualcoeff} and hence a generalization of \cite[Proposition 3]{NN15}. 
Geometrically,  this corollary (resp. Lemma \ref{lem:dualcoeff}) reflects the fact that  the stable bases of $K_T(T^*G/P_J)$ (resp. of $K_T(T^*G/B)$) corresponding to the opposite chambers are dual via the K-theory pairing.

\section{Relations with $p$-adic unramified principal series representations}\label{sec:padic}
In this section we compare the $K$-theory stable basis and the $T$-fixed point basis with certain bases in unramified principal series of $p$-adic groups. 

For the convenience of the readers, we also give a $K$-theory interpretation of the intertwiners, Macdonald's formula for the spherical functions \cite{M68,C80}, and the Casselman--Shalika formula for Whittaker functions \cite{CS80} from $p$-adic representations.
\subsection{Results from $p$-adic representations}

\subsubsection{Notations}
First, we recall some notions from $p$-adic representations, following \cite{R92}. 

Let  $F$ be a nonarchimedean local field, with ring of integers $\calO$,  a uniformizer $\varpi\in \calO$, and residue field $\bbF_q$.
Let $G_F$ be a split reductive  group over $F$, with maximal torus $A_F$ and Borel subgroup $B_F=A_FN_F$. Let $I$ be an Iwahori subgroup, i.e., the inverse image of $B(\bbF_q)$ under the evaluation map $G(\calO)\rightarrow G(\bbF_q)$. Note that the notations here differ from  \cite{R92}, where in {\it loc. cit.}, $B$ is used to denote the Iwahori subgroup, and $P$ denotes the Borel subgroup. To simplify notations, we let $\alpha, \beta$ denote the coroots of $G$. We also have the following decomposition 
\[G_F=\sqcup_{w\in W}B_FwI.\]

Let $\bbH=\bbC_c[I\backslash G_F/I]$ be the Iwahori Hekce algebra. It has two subalgebras, the finite Hecke algebra $H_W$, and a commutative subalgebra $\Theta$ which is isomorphic to the coordinate ring $\bbC[A^\vee]$ of the complex dual torus $A^\vee=\bbC^*\otimes X^*(A)$. More precisely,  $\Theta$  has a $\bbC$-linear basis $\{\theta_a\mid a\in A_F/A_\calO\}$. For any coroot $\alpha$ of $G$, let $h_{\alpha}:F^\times\rightarrow A_F$ be the corresponding one parameter subgroup. The isomorphism $\Theta\simeq\bbC[A^\vee]$ maps $\theta_{h_{\alpha}(\varpi)}$  to $e^{\alpha}\in X^*(A^\vee)\subset \bbC[A^\vee]$. So for any character $\tau$ of $A$, we have $e^{\alpha}(\tau)=\tau(h_{\alpha}(\varpi))$. We have the following pairing
\[\langle,\rangle: A_F/A_{\calO}\times A^\vee\rightarrow \bbC^*\] given by
\[\langle a,z\otimes \lambda \rangle=z^{\textit{val}(\lambda(a))}.\]
This induces an isomorphism between $A_F/A_{\calO}$ and the group $X^*(A^\vee)$ of rational characters of $A^\vee$. It also induces an identification between $A^\vee$ and unramified characters of $A$, i.e., characters which are trivial on $A_\calO$. As a $\bbC$-vector space, we have
\[\bbH=\Theta\otimes_\bbC H_W.\] 

Let $\tau$ be an unramified character of $A$ avoiding all the root hyperplanes. We consider the induced representation $I(\tau)=\Ind_B^G (\tau)$. As a $\bbC$-vector space, $\Ind_B^G (\tau)$ consists of  locally constant functions $f$ on $G_F$ such that $f(bg)=\tau(b)\delta^{\frac{1}{2}}(b)f(g)$ for any $b\in B_F$, where $\delta(b):=\prod_{\alpha>0}|\alpha^\vee(a)|_F$ is the modulus function on the Borel subgroup. The algebra $\bbH$ acts through convolution from the right on the Iwahori invariant subspace $I(\tau)^I$, so that the restriction of this action to $H_W$ is a regular representation. This right action is denoted by $\pi:\bbH\to \End_\bbC(I(\tau)^I)$.

\subsubsection{Interwiners}
For any character $\tau$ and $x\in W$, we can define $x\tau\in X^*(A)$ by the formula $x\tau(a):=\tau(x^{-1}ax)$ for any $a\in A$. Since we assume $\tau$ is unramified and has trivial stablizer under the Weyl group action, the space $\Hom_G(I(\tau), I(x^{-1}\tau))$ is one dimensional, spanned by an operator $\calA_x=\calA_x^\tau$ \footnote{This intertwiner $\calA_x$ is related to the one $T_x$ in \cite{C80} by the formula $\calA_x=T_{x^{-1}}$.} defined by 
\[\calA_x(\varphi)(g):=\int_{N_x}\varphi(\dot{x}ng)dn,\]
where $\dot{x}$ is a representative of $x\in W$, $N_x=N\cap \dot{x}^{-1}N^{-}\dot{x}$ with $N$ (resp. $N^{-}$) being the unipotent radical of the (opposite) Borel subgroup $B$, and the measure on $N_x$ is normalized by the condition that $\textit{vol}(N_x\cap G(\calO))=1$  \cite{R92}. If $x, y\in W$ satisfy $\ell(x)+\ell(y)=\ell(xy)$, then $\calA^{x^{-1}\tau}_y\calA^\tau_x=\calA_{xy}^{\tau}$. 

For any coroot $\alpha$, let 
\begin{equation}\label{equ:calpha}
c_\alpha=\frac{1-q^{-1}e^\alpha(\tau)}{1-e^\alpha(\tau)}.
\end{equation} 
We normalize the intertwiner as in \cite[Section 2.2]{HKP03} as follows:
\[I_w^\tau:=\prod_{\alpha>0,w^{-1}\alpha<0}\frac{1}{c_\alpha}\calA_w^\tau.\]
Then for any simple coroot $\alpha$ and any $y,w\in W$, we have
\begin{equation}\label{equ:prop of inter}
I_{s_\alpha}^{s_\alpha\tau}I_{s_\alpha}^\tau=1, \textit{\quad and \quad}  I_w^{y^{-1}\tau}I_y^\tau=I_{yw}^\tau.
\end{equation}

\subsubsection{Bases in Iwahori-invariants}
There are two bases of interest in $I(\tau)^I$.
One of these bases consists of normalized characteristic functions on the orbits, denoted by $\{\varphi_w^\tau\mid w\in W\}$. Here for any $w\in W$ the element $\varphi_w^\tau$ is characterized by the two conditions \cite[pg. 319]{R92}:
\begin{enumerate}
\item $\varphi_w^\tau$ is supported on $B_FwI$;
\item $\varphi_w^\tau(bwg)=\tau(b)\delta^{\frac{1}{2}}(b)$ for any $b\in B_F$ and $g\in I$.
\end{enumerate}
The action of  $\bbH$ on $I(\tau)^I$ has explicit formula under this basis \cite[pg. 325]{R92}. For any simple coroot $\alpha$, we have 
\begin{equation}\label{first}
\pi(T_{s_\alpha})(\varphi_w^\tau)=
\left\{ \begin{array}{cc}
q\varphi_{ws_\al}^\tau+(q-1)\varphi_w^\tau,& \text{ if } ws_{\alpha}<w;\\
\varphi_{ws_\al}^\tau,& \text{ if } ws_{\alpha}>w.
\end{array}\right.
\end{equation}
Under the intertwiner $I_{s_\alpha}^\tau$, this basis behaves as follows \cite[Theorem 3.4]{C80}
\begin{equation}\label{equ:intertwineronstandardbasis}
I_{s_\alpha}^\tau(\varphi_w^\tau)=
\left\{ \begin{array}{cc}
\frac{1}{qc_\alpha}\varphi_{s_\al w}^{s_\alpha\tau}+(1-\frac{1}{c_\alpha})\varphi_w^{s_\alpha\tau},& \text{ if } s_{\alpha}w>w;\\
\frac{1}{c_{\alpha}}\varphi_{s_\al w}^{s_\alpha\tau}+(1-\frac{1}{qc_\alpha})\varphi_w^{s_\alpha\tau},& \text{ if } s_\alpha w<w.
\end{array}\right.
\end{equation}

The second basis is called the Casselman's basis, denoted by $\{f_w^\tau\mid w\in W\}$. It consists of $\Theta$-eigenvectors in $I(\tau)^I$, and is further characterized in terms of the intertwining operators by the following formula\footnote{This basis is related to the one in \cite{C80} by an inversion of the index $w$.}
\[\calA_x^\tau (f_w^\tau)(1)=\delta_{x,w}.\]

The formula of the Hecke algebra action under this basis is also known \cite[Lemma 4.1 and Proposition 4.9]{R92}.
For any simple coroot $\alpha$ and $w\in W$, write
\[J_{\alpha,w}=
\left\{ \begin{array}{cc}
c_{w\alpha}c_{-w\alpha},& \text{ if } ws_{\alpha}>w;\\
1,& \text{ if } ws_{\alpha}<w.
\end{array}\right.  
\]
Then, we have 
\begin{equation}\label{second}
\pi(T_{s_\alpha})(f_w^\tau)=q(1-c_{w\alpha})f_w^\tau+qJ_{\alpha,w}f_{ws_\al}^\tau,
\end{equation}
\begin{equation}\label{eq:latticeaction1}
\pi(\theta_a)f_w^\tau=\tau(waw^{-1})f_w^\tau=(w^{-1}\tau(a))f^\tau_w, \hbox{ for any }a\in A,
\end{equation}
and (\cite[Theorem 4.2]{R92})
\begin{equation}\label{equ:intereigen}
I_{s_\alpha}^\tau(f_w^\tau)=
\left\{ \begin{array}{cc}
c_{-\alpha}f_{s_\alpha w}^{s_\alpha \tau},& \text{ if } s_{\alpha}w>w;\\
\frac{1}{c_{\alpha}}f_{s_\alpha w}^{s_\alpha \tau},& \text{ if } s_\alpha w<w.
\end{array}\right.
\end{equation}

\subsubsection{Transition matrices}
The change of bases matrix
\begin{equation}\label{equ:transition}
f_w^\tau=\sum_{y\geq w}a_{w,y}(\tau)\varphi_y^\tau
\end{equation}
is interesting (see, e.g., \cite{NN15}). It is clear that  $a_{w,w}=1$. A formula for the generating series of general $a_{w,y}$  is given in \cite[Proposition 5.2]{R92}, in terms of the canonical basis of \cite{KL79}. The inverse matrix $b_{w,y}(\tau)$ defined by 
\begin{equation}\label{equ:transition2}
\varphi_w^\tau=\sum_{y\geq w}b_{w,y}(\tau)f_y^\tau
\end{equation}
plays an important role in explicit computations of the Whittaker function on $\varphi_w^\tau$ (see, e.g.,  \cite{R93}). 

In Corollary \ref{cor:relation bet transitions}, we explain how 
Theorem~\ref{thm:intro2} gives a closed formula for these two matrices.

\subsubsection{Macdonald's formula for spherical function}
In this section, we review the Macdonald's formula for spherical functions.

According to the Iwasawa decomposition $G=BG(\calO)$, the vector space $I(\tau)^{G(\calO)}$ is one dimensional. Let $\phi^\tau$ be the basis normalized by the condition that $\phi^\tau(1)=1$. Then we have (\cite{C80})
\[\phi^\tau=\sum_w\varphi_w^\tau=\sum_w\prod_{\alpha>0,w^{-1}\alpha<0}\frac{1-q^{-1}e^\alpha(\tau)}{1-e^\alpha(\tau)}f_w^\tau.\]
It follows from either  \eqref{equ:intertwineronstandardbasis} or \eqref{equ:intereigen} that
\begin{equation}\label{equ:Gindikin–Karpelevich Formula}
I_w^\tau(\phi^\tau)=\phi^{w^{-1}\tau}.
\end{equation}
This formula is refereed to as the Gindikin--Karpelevich formula in literature.

We consider a sesquilinear paring  $\angl{-,-}:I(\tau^{-1})\otimes I(\tau)\to \bbC$ \cite[\S~1.9]{HKP03}.
For any $g\in G(F)$, we consider the following matrix coefficient
\begin{equation}\label{equ:definitionofspherical}
\Gamma_\tau(g)=\langle g\cdot \phi^\tau, \phi^{\tau^{-1}}\rangle.
\end{equation}
It satisfies
\[\Gamma_\tau(1)=1, \Gamma_\tau=\Gamma_{w\tau},\]
and 
\[\Gamma_\tau(k_1gk_2)=\Gamma_\tau(g)\]
for any $k_1, k_2\in G(\calO)$, and $g\in G$. 
This gives a well-defined $\bbC$-valued function on $G(\calO)\backslash G(F)/G(\calO)$.
This function $\Gamma_\tau$ is called the zonal spherical function corresponding to $\tau$. 

Let $X_*(A)_+$ be the dominant coweights. 
By the Cartan decomposition,
\[G(F)=\sqcup_{\mu\in X_*(A)_+} G(\calO)\varpi^\mu G(\calO),\] 
in order to know this function, it sufficed to know the value of $\Gamma_\tau$ at the $\varpi^\mu$'s, where $\varpi^\mu=h_\mu(\varpi)$.

For any dominant coweight $\mu$ of $G$, the characteristic function $1_{I\varpi^\mu I}$ is an element in the affine Hecke algebra $\bbH=\bbC_c[I\backslash G/I]$. Let
\begin{equation}\label{equ:defofcharacteristic}
e_{I\varpi^\mu I}=\frac{1_{I\varpi^\mu I}}{\textit{vol}(I\varpi^\mu I)}=\delta_{B}(\varpi^\mu)^{\frac{1}{2}}\theta_\mu\in \bbH.
\end{equation}
Then by the definition of $\Gamma_\tau$, we have
\begin{equation}\label{equ:newdefofspherical}
\Gamma_\tau(\varpi^\mu)=\langle \pi(e_{I\varpi^\mu I})(\phi^\tau), \phi^{\tau^{-1}}\rangle.
\end{equation}

The following is the Macdonald formula.
\begin{theorem}\cite{M68}\label{thm:macdonald}
Let $Q$ be the volume of $Bw_0B$ and $\mu\in X_*(A)_+$. We have
\[\Gamma_\tau(\varpi^\mu)=\frac{\delta_{B}^{\frac{1}{2}}(\varpi^\mu)}{Q}\sum_{w\in W}e^{w\mu}(\tau)\prod_{\beta>0}\frac{1-q^{-1}e^{-w\beta}(\tau)}{1-e^{-w\beta}(\tau)},\]
where $\delta_B$ is the modulus function on the Borel subgroup $B$. 
\end{theorem} 	
The proof by Casselman \cite{C80} uses the $W$-invariance of the function and the eigenbasis $f_w^\tau$.  This formula gives the Satake transform for the spherical Hecke algebra \cite[Theorem~5.6.1]{HKP03}.

In Theorem \ref{thm:kmacdonald}, we give an equivariant K-theoretic interpretation of this formula.
\subsubsection{Casselman--Shalika formula}
In this section, we review the Casselman--Shalika formula for the Whittaker functions, see \cite{CS80}.

Recall $N$ is the unipotent radical of the Borel subgroup $B$, and $\prod_{\alpha\in \Pi}N_\alpha$ is a quotient of $N$, where the product runs over all simple roots, and $N_\alpha$ is the corresponding root subgroup, all of which are isomorphic to the additive group. Given characters $\sigma_\alpha$ of $N_\alpha$, the product $\sigma:=\prod\sigma_\alpha$ is a character of $N$.  We say $\sigma$ is principle if all the $\sigma_\alpha$ are non-trivial. We say $\sigma$ is unramified if all the characters $\sigma_\alpha$ are trivial on $\calO$, but nontrivial on $\varpi^{-1}\calO$. From now on, we assume $\sigma$ is principal and unramified. Let $\Ind_N^G\sigma$ be the induced representation.

For every unramified character $\tau$ as before, a \textit{Whittaker functional} on $I(\tau)$ is a $\bbC$-module map 
\[L:I(\tau)\rightarrow \bbC,\]
such that $L(n\phi)=\sigma(n)L(\phi)$ for any $n\in N$ and $\phi\in I(\tau)$. It is proved in \cite{R73} (see also \cite{CS80}) that the space of Whittaker functional is one-dimensional. For any $f\in I(\tau)$, define $\calW_\tau(f):G\rightarrow \bbC$ by 
\[\calW_\tau(f)(g):=L(gf).\]
Then $\calW_\tau(f)$ is a function on $G$ satisfying 
\[\calW_\tau(f)(ng)=\sigma(n)\calW_\tau(f)(g), \textit{\quad if \quad} n\in N.\]
And $f\mapsto \calW_\tau(f)$ is a $G$-map from $I(\tau)$ to $\Ind_N^G\sigma$, denoted by $\calW_\tau$.\footnote{This is the notation used in \cite{R93}. We normalize our $L$ such that our $\calW_\tau(f)(g)$ coincides with the one in \textit{loc. cit.}.} It follows from \cite[Proposition 2.1]{CS80} that for fixed $g\in G$ and $f\in I(\tau)$, the function $\tau\rightarrow \calW_\tau(f)(g)$ is a polynomial function on the dual torus $A^\vee$.

The Whittaker functional $\calW_\tau$ enjoys the following properties (see \cite[Equation (1.3), Proposition 3.1]{R93}) :
\begin{equation}\label{equ:intertwinerwhittaker}
\calW_{w^{-1}\tau}I_w^\tau=\prod_{\beta>0,w^{-1}\beta<0}\frac{1-q^{-1}e^{-\beta}(\tau)}{1-q^{-1}e^\beta(\tau)}\calW_\tau,
\end{equation}
and for every dominant coweight $\mu$, 
\begin{equation}\label{equ:whittakereigenbasis}
\calW_\tau(f_w^\tau)(\varpi^\mu)=\delta_B^{\frac{1}{2}}(\varpi^\mu)e^{w\mu}(\tau)\prod_{\beta>0,w^{-1}\beta>0}\frac{1-q^{-1}e^\beta(\tau)}{1-e^{-\beta}(\tau)}.
\end{equation}

Recall we have the spherical function $\phi^\tau\in I(\tau)^{G(\calO)}$. We define the Whittaker function 
\[W_\tau(g):=\calW_\tau(\phi^\tau)(g)=L(g\phi^\tau).\]
The Casselman--Shalika formula is an explicit formula for $W_\tau$. Since $W_\tau$ is right $G(\calO)$-invariant, and for any $n\in N$,
\[W_\tau(ng)=\sigma(n)W_\tau(g),\]
we only need to determine the value of $W_\tau$ at the the elements $\varpi^\mu$ for any coweight $\mu\in X_*(A)$. Moreover,  $W_\tau(\varpi^\mu)=0$, unless $\mu$ is dominant, cause if not, there exists some $x\in N_\alpha\cap G(\calO)$, such that $\sigma_\alpha(\varpi^\mu x\varpi^{-\mu})$ is nontrivial. However,
\[W_\tau(\varpi^\mu)=W_\tau(\varpi^\mu x)=\sigma_\alpha(\varpi^\mu x\varpi^{-\mu})W_\tau(\varpi^\mu),\]
forcing $W_\tau(\varpi^\mu)=0$.

Assume $\mu$ is dominant. By the Iwahori factorization $I=(I\cap \bar{B})(I\cap N)$, we have $I\varpi^{-\mu} I=I\varpi^{-\mu} (I\cap N)$. Since $\sigma$ is trivial on $I\cap N$ and $\phi^\tau$ is invariant under $I$, we have \cite[Theorem 6.5.1]{HKP03}
\begin{equation}\label{equ:newdefofwhittaker}
W_\tau(\varpi^\mu)=L(\pi(e_{I\varpi^{-\mu} I})\phi^\tau).
\end{equation}

The Casselman--Shalika formula is given by the following theorem.
\begin{theorem}\cite[Theorem 5.4]{CS80}\label{casselmanshalika}
Let $\mu$ be a dominant coweight of $G$, then 
\begin{align*}
W_\tau(\varpi^\mu)&=\delta_B^{\frac{1}{2}}(\varpi^\mu)\prod_{\beta>0}(1-q^{-1}e^\beta(\tau))\sum_w\frac{e^{w\mu}(\tau)}{\prod_{\beta>0}(1-e^{-w\beta}(\tau))}\\
&=\delta_B^{\frac{1}{2}}(\varpi^\mu)\prod_{\beta>0}(1-q^{-1}e^\beta(\tau))E_\mu(\tau),
\end{align*}
where $E_\mu$ is the character of the representation of the Langlands dual group $G^\vee$ having highest weight $\mu$.
\end{theorem}

\subsection{Bases in equivariant K-theory for the complex dual group $G^\vee$} From now on we only consider K-theory with $\bbC$ coefficients. The Iwahori-Hecke algebra $\bbH$ (with $\bbC$ coefficients) of $G_F$ can be expressed in terms of the complex reductive group  $G^\vee$, whose root datum is Langlands dual to that of $G_F$.  
We adapt our notation  of stable basis from \S~\ref{sec:Heckeaction} to the group $G^\vee$. To simplify notations, we still let $\alpha,\beta$ denote the roots of $G^\vee$. The maximal torus of $G^\vee$ is naturally isomorphic to $A^\vee$, the complex dual of $A_F\subset G_F$.
Let $B^\vee$ be a Borel subgroup of $G^\vee$ containing $A^\vee$. In the remaining parts of the paper, \textbf{we switch the notation for the positive roots and negative roots}. That is,  we call all the roots in $B^\vee$ negative roots.

Recall we have
\[\bbH\simeq K_{G^\vee\times \bbC^*}(Z),\]
where $Z$ is the Steinberg variety for the dual group. In this section, we normalize this isomorphism the same way  as in \cite[Prop. 6.1.5]{R08} to better suite for the comparison. 
Under the modified  isomorphism \cite[Prop. 6.1.5]{R08},  $e^\lambda\in X^*(A^\vee)$ is mapped to  $\pi_\Delta^*(\calO(\lambda))$ with $\calO(\lambda)$ being the line bundle on $T^*(G^\vee/B^\vee)$; $\pi_\Delta:Z_{\Delta}=\Delta(T^*(G^\vee/B^\vee))\rightarrow T^*(G^\vee/B^\vee)$; the operator $T_\alpha\in \bbH$ for simple root $\alpha$ is mapped to 
\[-[\calO_\Delta]-[\calO_{T_{Y_\alpha}^*}(-\rho,\rho-\alpha)].\] 
Note that this action is different than the one used in \S~\ref{sec:Heckeaction}, hence the formulas from \S~\ref{sec:Heckeaction} are modified correspondingly  as below. 

The above isomorphism has symmetry$[\calO_{T_{Y_\alpha}^*}(-\rho,\rho-\alpha)]\simeq[\calO_{T_{Y_\alpha}^*}(\rho-\alpha,-\rho)]$ \cite[Lemma 1.5.1]{R08}. Hence, the right convolution and the left convolution by $T_\alpha$ will give the same operator on $K_{A^\vee\times \bbC^*}(T^*(G^\vee/B^\vee))$. 
In what follows, the right convolution action of $\bbH$ on $K_{A^\vee\times \bbC^*}(T^*(G^\vee/B^\vee))$ will be denoted by $\pi$. 
As in \S~\ref{sec:Heckeaction}, we use $T_\alpha$ (resp. $T_\alpha'$) to denote the left (resp. right) convolution by the simple generator of $\bbH$. The relations between these convolution operators are
\begin{equation}\label{equ:relationbetwdifferentoperators}
\pi(T_\alpha)=\calO(-\rho)T_\alpha\calO(\rho)=\calO(\rho)T_\alpha'\calO(-\rho),
\end{equation}
where $\calO(\pm\rho)$ is the operator of multiplication by the line bundle. Note that the second equality also gives a geometric proof of the first equality in \eqref{equ:relationwithlus}. 

For the bases of $K_{G^\vee\times \bbC^*}(Z)$, we will then consider the following instead:
\[
(\iota_{w*}1)_{-\rho}:=\iota_{w*}1\otimes  \calO(-\rho), \quad (\stab_-(w))_{-\rho}:=\stab_-(w)\otimes \calO(-\rho). 
\]
The fixed point basis is an eigenbasis for the action of the lattice part $\Theta$ of $\bbH$. Therefore, for any $e^\lambda\in X^*(A^\vee)$, 
\begin{align}\label{eq:latticeaction2}
\pi(e^\lambda)(\iota_{w*}1\otimes\calO(-\rho))=e^{w\lambda}\iota_{w*}1\otimes\calO(-\rho). 
\end{align}
From the proof of Lemma \ref{lem:geoadjoint} and Equation (\ref{equ:relationbetwdifferentoperators}) (we switch the positive and negative roots), we have
\begin{equation}\label{fixed}
\pi(T_{s_\alpha})(\iota_{w*}1)_{-\rho}= T_{s_\alpha}(\iota_{w*}1)\otimes \calO(-\rho)=
\frac{q-1}{1-e^{-w\alpha}}(\iota_{w*}1)_{-\rho}+\frac{q-e^{-w\alpha}}{1-e^{w\alpha}}(\iota_{ws*}1)_{-\rho}.
\end{equation}
As for the second basis, we get from Theorem \ref{thm:stabpim} the following
\begin{equation}\label{equ:newaction}
\pi(T_{s_\alpha})(\stab_-(w)_{-\rho})=
\left\{ \begin{array}{cc}
q^{\frac{1}{2}}(\stab_-(ws_{\alpha}))_{-\rho}+(q-1)(\stab_-(w)_{-\rho},& \text{ if } ws_{\alpha}<w;\\
q^{\frac{1}{2}}(\stab_-(ws_{\alpha}))_{-\rho},& \text{ if } ws_{\alpha}>w.
\end{array}\right. \\
\end{equation}

By definition, the fixed point basis $\iota_{w*}1$ is supported at $w$ with restriction
\[\iota_{w*}1|_w=\prod_{\beta>0}(1-e^{-w\beta})(1-qe^{w\beta}).\]
Hence, by the definition of $\stab_-(w)$ (Theorem \ref{thm:geostable}), the second part of Remark \ref{rem:afterdef} and Lemma \ref{lem:charstab}, we can write
\begin{equation}\label{equ:transition3}
\iota_{w*}1=q^{-\frac{\ell(w)}{2}}\prod_{\beta>0, w\beta>0}(1-e^{-w\beta})\prod_{\beta>0,w\beta<0}(q-e^{-w\beta})\stab_-(w)+\sum_{y>w}\stab_+(y)|_w\stab_-(y),
\end{equation}
where $\stab_+(y)|_w$ is given by Theorem \ref{thm:intro2}.

Since an unramified character $\tau$ of $A_F$ corresponds to a maximal ideal in $K_{A^\vee}(pt)$, we have the evaluation map $K_{A^\vee}(pt)\to \bbC_\tau$. Consequently, we have the tensor product 
\[K_{A^\vee\times \bbC^*}(T^*(G^\vee/B^\vee))\otimes_{K_{A^\vee}(pt)}\bbC_\tau\] 
which without raising any confusion will also be denoted by \[K_\tau:=K_{A^\vee\times \bbC^*}(T^*(G^\vee/B^\vee))\otimes_{K_{A^\vee\times \bbC^*}(pt)}\bbC_{\tau}.\]
For any $f\in K_{A^\vee\times \bbC^*}(T^*(G^\vee/B^\vee))$, the corresponding class $f\otimes 1\in K_\tau $ will also be denoted by $f$ for simplicity.
We further assume that the values of the roots of $G^\vee$ at $\tau$ does not equal to $q^{\pm1}$, so that the above tensor product has the following two bases
\[\{(\iota_{w*}1)_{-\rho}\mid w\in W\} \text{\quad and \quad} \{(\stab_-(w))_{-\rho}\mid w\in W\}. \] 

\subsection{The comparison}
The main result of this section is the following.
\begin{theorem}\label{thm:comparison}
Fix an unramified character $\tau$ of $A$. There is a unique isomorphism between the following two right $\bbH$-modules 
\[\Psi:K_{A^\vee\times \bbC^*}(T^*(G^\vee/B^\vee))\otimes_{K_{A^\vee\times \bbC^*}(pt)}\bbC_{\tau}\rightarrow I(\tau)^I,\] 
with the equivariant parameter $q$ for $\bbC^*$ evaluated to the cardinality of the residue field of $\calO_F$, satisfying the following properties: 
\begin{enumerate}
\item for any $w\in W$,
\[g_w:=\frac{q^{\ell(w)}}{\prod_{\beta>0, w\beta>0}(1-e^{-w\beta})\prod_{\beta>0,w\beta<0}(q-e^{-w\beta})}(\iota_{w*}1)_{-\rho} \quad \mapsto \quad f_w^\tau, \]
\item 
$(\stab_-(w))_{-\rho} \mapsto q^{-\frac{\ell(w)}{2}}\varphi_w^\tau$.
\end{enumerate}
\end{theorem}
\begin{remark}\label{rem:aftercomparison}
\begin{enumerate}
\item 
Such an isomorphism has been studied by Lusztig \cite{Lus98} and Braverman--Kazhdan \cite{BK99} from different points of view.  
 However, the present paper explicitly identity different bases from K-theory and from $p$-adic representation theory, which had been previously unknown.
\item 
Under this isomorphism, the spherical function $\phi^\tau$ corresponds to the following element on the K-theory side:
\begin{align}\label{equ:kspherical}
\tilde{\phi}^\tau:&=\sum_{w}q^{\frac{\ell(w)}{2}}(\stab_-(w))_{-\rho}=\sum_w\frac{(\iota_{w*}1)_{-\rho}}{\prod_{\beta>0}(1-e^{-w\beta})}\\
=&[\calO_{G^\vee/B^\vee}\otimes\calO(-\rho)]\in K_{A^\vee\times \bbC^*}(T^*(G^\vee/B^\vee))\otimes_{K_{A^\vee\times \bbC^*}(pt)}\bbC_{\tau},
\end{align}
where the last equality follows from localization. Hence, in what follow we refer to $\tilde{\phi}^\tau$ as the K-theoretic spherical class.
\end{enumerate}
\end{remark}
\begin{proof}
Condition (1) uniquely defines $\Psi$ as a map of $\bbC$-vector spaces.  We need to check that $\Psi$ is a map of $\bbH$-modules, and that it is satisfies Condition (2).

First we verify that $\Psi$ is a map of $\bbH$-modules. For any simple root $\alpha$, by \eqref{fixed}, we have
\begin{align*}
&\pi(T_{s_\al})(g_w)\\
=&\frac{q^{\ell(w)}}{\prod_{\beta>0, w\beta>0}(1-e^{-w\beta})\prod_{\beta>0,w\beta<0}(q-e^{-w\beta})}\left(\frac{q-1}{1-e^{-w\alpha}}\iota_{w*}1\otimes\calO(-\rho)+\frac{q-e^{-w\alpha}}{1-e^{w\alpha}}\iota_{ws*}1\otimes\calO(-\rho)\right)\\
=&\frac{q-1}{1-e^{-w\alpha}}g_w
+\left\{\begin{array}{cc}
qg_{ws_\alpha},& \text{ if } ws_{\alpha}<w;\\
\frac{q-e^{-w\alpha}}{1-e^{w\alpha}}\frac{q-e^{w\alpha}}{1-e^{-w\alpha}}q^{-1}g_{ws_\alpha},& \text{ if } ws_{\alpha}>w.
\end{array}\right.
\end{align*}
Applying the map $\Psi$, we get
\begin{align*}
&\Psi\left(\pi(T_{s_\al})(g_w)\right)\\
=&\frac{q-1}{1-e^{-w\alpha}(\tau)}f_w^\tau+\left\{\begin{array}{cc}qf_{ws}^\tau,& \text{ if } ws_{\alpha}<w;\\
\frac{(1-q^{-1}e^{-w\alpha}(\tau))}{1-e^{w\alpha}(\tau)}\frac{1-q^{-1}e^{w\alpha}(\tau)}{1-e^{-w\alpha}(\tau)}qf_{ws}^\tau,& \text{ if } ws_{\alpha}>w.
\end{array}\right.\\
=&q(1-c_{w\alpha})f_w^\tau+qJ_{\alpha,w}f_{ws}^\tau\\
=&\pi(T_{s_\al})(f_w^\tau)\\
=&\pi(T_{s_\al})(\Psi(g_w)).
\end{align*}
Therefore, $\Psi$ commutes with the $H_W$-actions. Next we consider the action of $\Theta$. For any $e^\lambda\in X^*(A^\vee)$ We have
\begin{align*}
\Psi((\pi(e^\lambda)g_w)=\Psi(e^{w\lambda}g_w)=e^{w\lambda}(\tau)f_w^\tau=\pi(\theta_\lambda)f_w^\tau=\pi(\theta_\lambda)\Psi(g_w).
\end{align*}
This proves that $\Psi$ is a map of $\bbH$-modules.

We now prove Condition (2) by descending induction on $\ell(w)$. For the longest element $w=w_0\in W$, we have $f_{w_0}^\tau=\varphi_{w_0}^\tau$, and 
\[\iota_{w_0*}1=q^{-\frac{\ell(w_0)}{2}}\prod_{\beta>0, w_0\beta>0}(1-e^{-w_0\beta})\prod_{\beta>0,w_0\beta<0}(q-e^{-w_0\beta})\stab_-(w_0).\] Therefore, 
\[(\stab_-(w_0))_{-\rho}=q^{-\frac{\ell(w_0)}{2}}g_{w_0}.\] This proves Condition (2) for $w=w_0$. The inductive step follows directly from \eqref{first} and \eqref{equ:newaction}.
\end{proof}

One immediate corollary is the following relation between the restriction formulas in Theorem \ref{thm:intro1} and the transition matrix in Equations (\ref{equ:transition}) and (\ref{equ:transition2}). 
\begin{corollary}\label{cor:relation bet transitions}
For any $y,w\in W$ with $w\leq y$, we have
\begin{equation}\label{equ:match of transitions}
\stab_+(y)|_w=q^{\frac{\ell(y)}{2}-\ell(w)}a_{w,y}\prod_{\beta>0, w\beta>0}(1-e^{-w\beta})\prod_{\beta>0,w\beta<0}(q-e^{-w\beta}),
\end{equation}
and 
\begin{equation}\label{equ:match of transitionss}
\stab_-(w)|_y=q^{\ell(y)-\frac{\ell(w)}{2}}b_{w,y}\prod_{\beta>0, y\beta>0}(1-qe^{y\beta})\prod_{\beta>0,y\beta<0}(1-e^{y\beta}).
\end{equation}
\end{corollary}
\begin{proof}
By the definition of $a_{w,y}$ in Equation (\ref{equ:transition}) and the above theorem, we have
\[g_w=\sum_{z}a_{w,z}q^{\frac{\ell(z)}{2}}(\stab_-(z))_{-\rho}.\]
Pairing with $\stab_+(y)\otimes\calO(\rho)$ on both sides and using the duality between the opposite stable bases (see Remark \ref{rem:afterdef}), we get the first equation. The other equation follows immediately from the theorem and the localization formula.
\end{proof}

\subsection{Weyl group action and intertwiners}
In this section, we compare, under Theorem \ref{thm:comparison}, the Weyl group action on the equivariant K-theory side, and the intertwiner action $I_x^\tau$ on the $p$-adic side.
Recall for any $G^\vee\times \bbC^*$-variety $Y$, we have a Weyl group  action on $K_{A^\vee\times \bbC^*}(Y)$ defined as follows. For any $w\in W$, pick an representative $\dot{w}\in N_{G^\vee}(A^\vee)$, the normalizers of $A^\vee$ in $G^\vee$. Then left multiplication by $\dot{w}^{-1}$ defines a morphism from $Y$ to itself, which is not $A^\vee\times \bbC^*$-equivariant. For any $\calF\in K_{A^\vee\times \bbC^*}(Y)$, the pullback sheaf $(\dot{w}^{-1})^*\calF$ has a natural $A\times \bbC^*$-equivariant coherent sheaf structure. Hence $(\dot{w}^{-1})^*\calF\in K_{A^\vee\times \bbC^*}(Y)$, and this construction does not depend on the choice of the representative. So we get a Weyl group action on $K_{A^\vee\times \bbC^*}(Y)$. 
Note that $W$ acts on  the base ring $K_{A^\vee\otimes \bbC^*}(pt)$ by $w(e^\lambda)=e^{w\lambda}$ for any $e^\lambda \in K_{A^\vee\otimes \bbC^*}(pt)$. The action on $K_{A^\vee\times \bbC^*}(Y)$ makes it a $W$-equivariant $K_{A^\vee\otimes \bbC^*}(pt)$-module.

More explicitly, for $Y=T^*(G^\vee/B^\vee)$, the action can be written under localization as follows. For any $\calF\in K_{A^\vee\times \bbC^*}(T^*(G^\vee/B^\vee))$, we have
\[w(\calF)|_y=w(\calF|_{w^{-1}y}).\]
In particular, we have
\[w(\iota_{y*}1)=\iota_{wy*}1.\]

Let us use $\Psi_\tau$ to denote the isomorphism $\Psi$ in Theorem \ref{thm:comparison}. Then we have the following compatibility result, which is also studied by Braverman--Kazhdan from a different point of view \cite[Corollary 5.7]{BK99}.
\begin{corollary}\label{cor:weylintertwiner}
For any $w\in W$, the following diagram is commutative
\[
\xymatrix{K_\tau\ar[r]^-{\Psi_\tau}\ar[d]^{w\otimes 1} & I(\tau)^I\ar[d]^{I_{w^{-1}}^\tau}\\
K_{w\tau} \ar[r]^-{\Psi_{w\tau}}& I(w\tau)^I.}
\]
\end{corollary}
Note that the map $w\otimes 1$ is well-defined. For any $\calF\in K_{A^\vee\times \bbC^*}(T^*(G^\vee/B^\vee))$, $e^\lambda\in K_{A^\vee\times \bbC^*}(pt)$ and $z\in \bbC_\tau$, $e^\lambda\calF\otimes z=\calF\otimes ze^{\lambda}(\tau)$. And $(w\otimes 1)(e^\lambda\calF\otimes z)=w(e^\lambda\calF)\otimes z=e^{w\lambda}w(\calF)\otimes z=w(\calF)\otimes ze^{w\lambda}(\tau)=(w\otimes 1)(\calF\otimes ze^{\lambda}(\tau))$. 

\begin{proof}
By the properties of the intertwiners $I_w^\tau$ \eqref{equ:prop of inter}, we only  need to prove the corollary in the case when $w$ is a simple reflection $s_\alpha$. Using the notations in Theorem \ref{thm:comparison}, we check the commutativity using the basis $g_w$.

Fist of all, we have the following easy identities
\begin{align*}
&\prod_{\beta>0, s_\alpha w\beta>0}(1-e^{-s_\alpha w\beta}(s_\alpha\tau))\prod_{\beta>0,s_\alpha w\beta<0}(q-e^{-s_\alpha w\beta}(s_\alpha\tau))\\
=&\prod_{\beta>0, s_\alpha w\beta>0}(1-e^{-w\beta}(\tau))\prod_{\beta>0,s_\alpha w\beta<0}(q-e^{-w\beta}(\tau))\\
=&\left\{ \begin{array}{cc}
qc_{-\alpha}\prod_{\beta>0, w\beta>0}(1-e^{-w\beta}(\tau))\prod_{\beta>0,w\beta<0}(q-e^{-w\beta}(\tau)),& \text{ if } w^{-1}\alpha>0;\\
\frac{1}{qc_\alpha}\prod_{\beta>0, w\beta>0}(1-e^{-w\beta}(\tau))\prod_{\beta>0,w\beta<0}(q-e^{-w\beta}(\tau)),& \text{ if } w^{-1}\alpha<0,
\end{array}\right.
\end{align*}
where $c_\alpha$ is defined in \eqref{equ:calpha}. For example, if $w^{-1}\alpha>0$, then 
\[\prod_{\beta>0, s_\alpha w\beta>0}(1-e^{-w\beta})=\prod_{\beta\in R^+\setminus \{w^{-1}\alpha\}, s_\alpha w\beta>0}(1-e^{-w\beta})=\prod_{\beta\in R^+\setminus \{w^{-1}\alpha\}, w\beta>0}(1-e^{-w\beta})=\frac{1}{1-e^{-\alpha}}\prod_{\beta>0, w\beta>0}(1-e^{-w\beta}).\]
Notice that $w^{-1}\alpha>0$ iff $s_\alpha w>w$, and $w^{-1}\alpha<0$ iff $s_\alpha w<w$.

Using these, we have
\begin{align*}
(s_\alpha\otimes 1)(g_w)&=(s_\alpha\otimes 1)\left(q^{\ell(w)}\iota_{w*}1\otimes \frac{e^{-w\rho}(\tau)}{\prod_{\beta>0, w\beta>0}(1-e^{-w\beta}(\tau))\prod_{\beta>0,w\beta<0}(q-e^{-w\beta}(\tau))}\right)\\
=&q^{\ell(w)}\iota_{s_\alpha w*}1\otimes \frac{e^{-s_\alpha w\rho}(s_\alpha\tau)}{\prod_{\beta>0, w\beta>0}(1-e^{-w\beta}(\tau))\prod_{\beta>0,w\beta<0}(q-e^{-w\beta}(\tau))}\\
=&\left\{ \begin{array}{cc}
g_{s_\alpha w}\otimes c_{-\alpha},& \text{ if } w^{-1}\alpha>0;\\
g_{s_\alpha w}\otimes \frac{1}{c_\alpha},& \text{ if } w^{-1}\alpha<0.
\end{array}\right.\\
\in& K_{s_\alpha\tau}.
\end{align*}
Comparing with \eqref{equ:intereigen}, we get 
\[\Psi_{s_\alpha\tau}(s_\alpha\otimes 1)(g_w)=I_{s_\alpha}^\tau\left(\Psi_\tau(g_w)\right),\]
which finishes the proof.

\end{proof}
\begin{remark}
As an application of this corollary, we reprove the Gindikin--Karpelevich formula \eqref{equ:Gindikin–Karpelevich Formula}. 
According to Remark~\ref{rem:aftercomparison}, the two sides of the Gindikin--Karpelevich formula
 becomes 
$w(\tilde{\phi}^\tau)$ and $\tilde{\phi}^{w\tau}$,
where $\tilde{\phi}^\tau$ is the K-theoretic spherical vector defined in \eqref{equ:kspherical}. The equality between these two $K$-theory classes follows directly from the definition of $w\otimes 1$.
\end{remark}
\subsection{Macdonald's formula in equivariant K-theory}
In this section, we give a $K$-theory interpretation of the Macdonald's formula from Theorem \ref{thm:macdonald}.

We define the $K$-theory analogue of the pairing
$\angl{-,-}:I(\tau)\otimes I(\tau^{-1})\to \bbC$ \cite[\S~1.9]{HKP03}
to be the following. Let $\iota:A^\vee\to A^\vee$ be the endomorphism of abelian groups sending an element to its inverse. It induced a map $\iota:K_{A^\vee}(Y)\to K_{A^\vee}(Y)$ for any $A^\vee$-variety $Y$. 
Explicitly, on $K_{A^\vee\times \bbC^*}(T^*(G^\vee/B^\vee))$, using localization we have $\iota(\calF)|_w(\tau)=\calF|_w(\tau^{-1})$, for any point $\tau\in A^\vee$. 
We consider the paring \[\langle -, -\rangle_\tau: K_{\tau} \times K_{\tau^{-1}}\rightarrow \bbC \] defined as \[(\calF,\calG)\mapsto \left(p_*\iota(\sHom(\calF,\calG))\right)(\tau),\] with  $p:T^*(G^\vee/B^\vee)\to \fg^\vee$ being the Springer map. Here $(-)(\tau)$ means evaluating the $K$-theory classes using the map $K_{A^\vee\times\bbC^*}(\pt)\to\bbC_\tau$ induced by the character $\tau$ and the $\bbC^*$-equivariant parameter $q$ is evaluated to be the cardinality of the residue field of $\calO_F$. It is easy to see this pairing is well defined. Using localization, the above paring can be written as
\begin{equation}\label{equ:kpairingtwisted}
(\calF,\calG)\mapsto \sum_{w\in W}\frac{\calF|_w(\tau) \calG|_w(\tau^{-1})}{\bigwedge^\bullet T_w(T^*(G^\vee/B^\vee))(\tau^{-1})}.
\end{equation}
We will show in Remark~\ref{rmk:pairing} that under the isomorphism in Theorem~\ref{thm:comparison},  this paring  differs from $\angl{-,-}:I(\tau)\otimes I(\tau^{-1})\to \bbC$ \cite[\S~1.9]{HKP03} by a scalar.

For any $\mu$ in the dominant coweights $X_*(A)_+$ of $G$, we have the element \[e_{I\varpi^\mu I}=\delta_{B}(\varpi^\mu)^{\frac{1}{2}}\theta_\mu\in \bbH.\]
We define the $K$-theoretic analogue of the spherical function \eqref{equ:newdefofspherical} as follows 
\[\tilde{\Gamma}_\tau:X_*(A)_+\rightarrow \bbC,\quad \mu\rightarrow \langle \pi(e_{I\varpi^\mu I})\tilde{\phi}^\tau, \tilde{\phi}^{\tau^{-1}}\rangle_\tau,\]
where recall that  $\pi(e_{I\varpi^\mu I})$ is the action of $\Theta$ on the equivariant K-theory.

Then our K-theoretic interpretation of Macdonald's formula is the following
\begin{theorem}\label{thm:kmacdonald}
For any dominant coweight $\mu$ of $G$, we have
\[\tilde{\Gamma}_\tau(\mu)=q^{\dim G/B}Q\cdot \Gamma_\tau(\varpi^\mu),\]
where $Q$ is defined as the volume of $Bw_0B$.
\end{theorem}
\begin{remark}\label{rmk:pairing}
By the linear independence of the coweights $\mu$, it follows from the theorem that if we normalize the pairing in \eqref{equ:kpairingtwisted} by multiplying by $q^{\dim G/B}Q$, then the isomorphism $\Psi_\tau$ in Theorem \ref{thm:comparison} respects this pairing and the pairing between the contragredient modules $I(\tau)$ and $I(\tau^{-1})$.
\end{remark}
\begin{proof}
We use localization formula. First of all, we have
\begin{equation}\label{equ:spaneigen}
\pi(e_{I\varpi^\mu I})\tilde{\phi}^\tau=\delta_{B}(\varpi^\mu)^{\frac{1}{2}}\sum_we^{w\mu-w\rho}(\tau)\frac{(\iota_{w*}1)_{-\rho}\otimes 1}{\prod_{\beta>0}(1-e^{-w\beta}(\tau))}.
\end{equation} Then
\begin{align*}
\tilde{\Gamma}_\tau(\mu)&=\langle \pi(e_{I\varpi^\mu I})\tilde{\phi}^\tau, \tilde{\phi}^{\tau^{-1}}\rangle_\tau\\
&=\delta_{B}(\varpi^\mu)^{\frac{1}{2}}\sum_{w}e^{w\mu}(\tau)\frac{\bigwedge^\bullet T_w(T^*(G^\vee/B^\vee))(\tau)}{\prod_{\beta>0}(1-e^{-w\beta}(\tau))(1-e^{w\beta}(\tau))}\\
&=\delta_{B}(\varpi^\mu)^{\frac{1}{2}}\sum_{w}e^{w\mu}(\tau)\prod_{\beta>0}\frac{1-qe^{w\beta}(\tau)}{1-e^{w\beta}(\tau)}\\
&=q^{\dim G/B}\delta_{B}(\varpi^\mu)^{\frac{1}{2}}\sum_{w}e^{w\mu}(\tau)\prod_{\beta>0}\frac{1-q^{-1}e^{-w\beta}(\tau)}{1-e^{-w\beta}(\tau)}\\
&=q^{\dim G/B}Q\cdot \Gamma_\tau(\varpi^\mu).
\end{align*}
\end{proof}

\subsection{Casselman--Shalika formula in equivariant K-theory}
In this section, we investigate the K-theoretic meaning of the Casselman--Shalika formula, see Theorem \ref{casselmanshalika}.

First of all, we define the K-theoretic analogue of the Whittaker functional. On the space $K_{\tau}$, we have the equivariant Euler characteristic map
\[p_*:K_\tau\rightarrow \bbC\] induced by $p:T^*(G^\vee/B^\vee)\to \fg^\vee$.  Via localization,  for any $\calF\in K_\tau$, we have
\[p_*(\calF):=\sum_i (-1)^iH^i(T^*(G^\vee/B^\vee), \calF)(\tau)=\sum_{w\in W}\frac{\calF|_w(\tau)}{\bigwedge^\bullet(T_w(T^*G/B))(\tau)}\in \bbC.\]
In particular,  $p_*(\iota_{w*}1)=1$ for any $w\in W$.

The $K$-theoretic analogue of the Whittaker functional is defined to be a modification of the above
\[\tilde{L}_\tau(-):=\prod_{\beta>0}(1-q^{-1}e^\beta(\tau))\cdot p_*(-\otimes\calO(\rho)):K_\tau\rightarrow \bbC.\]
 For any $\calF\in K_\tau$ and any dominant coweight $\mu$ of $G$, we consider
\[\tilde{\calW}_\tau(\calF)(\varpi^\mu):=\tilde{L}_{\tau}\left(\pi(e_{I\varpi^\mu I})\calF\right),\]
where $e_{I\varpi^\mu I}$ is defined in \eqref{equ:defofcharacteristic}.
Here recall that
$e_{I\varpi^\mu I}=\delta_{B}(\varpi^\mu)^{\frac{1}{2}}\theta_\mu\in \bbH$, and
  $\pi(e_{I\varpi^\mu I})$ is the action of $\Theta$ on the equivariant $K$-theory. The
$K$-theory analogue of the Whittaker function is defined to be
\[\tilde{W}_\tau:X_*(A)_+\rightarrow \bbC,\quad \tilde{W}_\tau(\mu):=\tilde{L}_{\tau}\left(\pi(e_{I\varpi^\mu I})\tilde{\phi}^\tau\right),\]
where $\tilde\phi^\tau$ is our K-theoretic spherical class defined in \eqref{equ:kspherical}.

Under the isomorphism in Theorem \ref{thm:comparison}, properties of the Whittaker functions in \eqref{equ:intertwinerwhittaker} and \eqref{equ:whittakereigenbasis} correspond to the following lemma, 
according to Corollary \ref{cor:weylintertwiner}.

\begin{lemma}
For any $w\in W$, we have
\begin{enumerate}
\item
The following diagram is commutative
\[
\xymatrix{K_\tau\ar[r]^-{\tilde{L}_\tau}\ar[d]^{w^{-1}\otimes 1} & \bbC \ar[d]^{\prod_{\beta>0,w^{-1}\beta<0}\frac{1-q^{-1}e^{-\beta}(\tau)}{1-q^{-1}e^\beta(\tau)}}\\
K_{w^{-1}\tau} \ar[r]^-{\tilde{L}_{w^{-1}\tau}}& \bbC.}
\]
\item 
For the $g_w$ defined in Theorem \ref{thm:comparison}, we have
\[\tilde{\calW}_\tau(g_w)(\varpi^\mu)=\calW_\tau(f_w^\tau)(\varpi^\mu).\]
\end{enumerate}
\end{lemma}

\begin{proof}
\begin{enumerate}
\item
Since any $\calF\in K_\tau$ can be written as $\bbC$-linear combination of the fixed point basis $\iota_{y*}1\otimes 1\in K_\tau$, we only need to check for these basis elements. Then it follows from the following easy identity
\[\prod_{\beta>0}(1-q^{-1}e^\beta)\prod_{\beta>0,w^{-1}\beta<0}\frac{1-q^{-1}e^{-\beta}}{1-q^{-1}e^\beta}=\prod_{\beta>0}(1-q^{-1}e^{w\beta}).\]
\item
The second one is verified immediately once we know $\pi(e_{I\varpi^\mu I})(g_w)=\delta_B^{\frac{1}{2}}(\varpi^\mu)e^{w\mu}g_w$.
\end{enumerate}
\end{proof}

Then the following is the $K$-theoretic interpretation of the Casselman--Shalika formula.
\begin{theorem}\label{thm:kcasselmanshalika}
For any dominant coweight $\mu$ of $G$, we have
\[\tilde{W}_\tau(\mu)=W_\tau(\varpi^\mu).\]
\end{theorem}
\begin{proof}
From  \eqref{equ:spaneigen} and the fact $p_*(\iota_{w*}1)=1$, we get
\begin{align*}
\tilde{W}_\tau(\mu)&=\tilde{L}_{\tau}\left(\pi(e_{I\varpi^\mu I})\tilde{\phi}^\tau\right)\\
&=\delta_B^{\frac{1}{2}}(\varpi^\mu)\prod_{\beta>0}(1-q^{-1}e^\beta(\tau))\sum_{w\in W}e^{w\mu}(\tau)\frac{1}{\prod_{\beta>0}(1-e^{-w\beta}(\tau))}\\
&=W_\tau(\varpi^\mu).
\end{align*}
\end{proof}

\newcommand{\arxiv}[1]
{\texttt{\href{http://arxiv.org/abs/#1}{arXiv:#1}}}
\newcommand{\doi}[1]
{\texttt{\href{http://dx.doi.org/#1}{doi:#1}}}
\renewcommand{\MR}[1]
{\href{http://www.ams.org/mathscinet-getitem?mr=#1}{MR#1}}

\end{document}